\documentclass[platex]{amsart}
\usepackage{amssymb}
\usepackage{braket}
\usepackage{mathrsfs}
\usepackage{ifthen}
\usepackage{graphicx}
\usepackage{tikz}
\usetikzlibrary{patterns}
\usepackage{comment}
\usepackage{mleftright}
\usepackage{cleveref}
 \usepackage[all]{xy}
 \usepackage{amscd}
 \usepackage{amsrefs}
 \usepackage{color}
\usepackage{epstopdf}
\usepackage{cancel}

\theoremstyle{plain}
\newtheorem{theo}{Theorem}[section]
\crefname{theo}{Theorem}{Theorems}
\Crefname{theo}{Theorem}{Theorems}
\newtheorem{prop}[theo]{Proposition}
\crefname{prop}{Proposition}{Propositions}
\Crefname{prop}{Proposition}{Propositions}
\newtheorem{lem}[theo]{Lemma}
\crefname{lem}{Lemma}{Lemmas}
\Crefname{lem}{Lemma}{Lemmas}
\newtheorem{cor}[theo]{Corollary}
\crefname{cor}{Corollary}{Corollaries}
\Crefname{cor}{Corollary}{Corollaries}

\crefname{claim}{Claim}{Claims}
\Crefname{claim}{Claim}{Claims}

\crefname{property}{Property}{Properties}
\Crefname{property}{Property}{Properties}

\crefname{problem}{Problem}{Problems}
\Crefname{problem}{Problem}{Problems}

\theoremstyle{definition}
\newtheorem{defi}[theo]{Definition}
\crefname{defi}{Definition}{Definitions}
\Crefname{defi}{Definition}{Definitions}

\crefname{notation}{Notation}{Notations}
\Crefname{notation}{Notation}{Notations}

\crefname{convention}{Convention}{Conventions}
\Crefname{convention}{Convention}{Conventions}

\crefname{cond}{Condition}{Conditions}
\Crefname{cond}{Condition}{Conditions}

\crefname{assum}{Assumption}{Assumptions}
\Crefname{assum}{Assumption}{Assumptions}

\theoremstyle{remark}
\newtheorem{rem}[theo]{Remark}
\crefname{rem}{Remark}{Remarks}
\Crefname{rem}{Remark}{Remarks}

\crefname{ex}{Example}{Examples}
\Crefname{ex}{Example}{Examples}

\crefname{section}{Section}{Sections}
\Crefname{section}{Section}{Sections}
\crefname{subsection}{Subsection}{Subsections}
\Crefname{subsection}{Subsection}{Subsections}
\crefname{figure}{Figure}{Figures}
\Crefname{figure}{Figure}{Figures}
\newtheorem*{acknowledgment}{Acknowledgment}

\newcommand{\Z}{\mathbb{Z}}

\newcommand{\R}{\mathbb{R}}
\newcommand{\C}{\mathbb{C}}
\newcommand{\quat}{\mathbb{H}}
\newcommand{\Q}{\mathbb{Q}}

\newcommand{\fraks}{\mathfrak{s}}
\newcommand{\frakt}{\mathfrak{t}}

\newcommand{\tilR}{\tilde{\mathbb{R}}}

\newcommand{\ind}{\mathop{\mathrm{ind}}\nolimits}

\newcommand{\wt}{\widetilde}
\newcommand{\al}{\alpha}
\def\cU{\mathcal{U}}

\def\Pin{\operatorname{Pin}}

\def\s{\mathfrak{s}}
\def\t{\text}
\def\C{\mathbb{C}}

\def\D{\cancel{\partial}}
\def\spinc{\text{spin}^c}
\def\wt{\widetilde}

\newcommand{\G}{\mathcal G}

\newcommand{\pr}{\text{pr}}
\newcommand{\Om}{\Omega}

\newcommand{\del}{\partial}

\newcommand{\Ker}{\mathop{\mathrm{Ker}}\nolimits}
\newcommand{\Coker}{\mathop{\mathrm{Coker}}\nolimits}
\newcommand{\im}{\mathop{\mathrm{Im}}\nolimits}

\newcommand{\LE}{\mathcal{LE}}
\newcommand{\SWF}{\mathrm{SWF}}

\title[PSC and and homology cobordism invariants]{Positive scalar curvature and homology cobordism invariants}

\author{Hokuto Konno}
\address{Graduate School of Mathematical Sciences, the University of Tokyo, 3-8-1 Komaba, Meguro, Tokyo 153-8914, Japan}
\email{konno@ms.u-tokyo.ac.jp}

\author{Masaki Taniguchi}
\address{2-1 Hirosawa, Wako, Saitama 351-0198, Japan}
\email{masaki.taniguchi@riken.jp}

\begin{document}

\maketitle

\begin{abstract}
We determine the local equivalence class of the Seiberg--Witten Floer stable homotopy type of 
a spin rational homology 3-sphere $Y$ embedded into a spin rational homology $S^{1} \times S^{3}$ with a positive scalar curvature metric so that $Y$ generates the third homology.
The main tool of the proof is a relative Bauer--Furuta-type invariant on a periodic-end $4$-manifold.
As a consequence, we give obstructions to positive scalar curvature metrics on spin rational homology $S^{1} \times S^{3}$, typically described as the coincidence of various Fr{\o}yshov-type invariants.
This coincidence also yields alternative proofs of two known obstructions by Jianfeng Lin and by the authors for the same class of $4$-manifolds.
\end{abstract}

\tableofcontents

\section{Introduction}

\subsection{Seiberg--Witten Floer stable homotopy type and local equivalence}
Manolescu's Seiberg--Witten Floer stable homotopy type~\cite{Ma03} is a space-valued Floer theoretic invariant of a rational homology $3$-sphere equipped with a spin$^{c}$ structure, and recovers the monopole Floer homology defined by Kronheimer and Mrowka~\cite{KM07} for this class of $3$-manifolds~\cite{LM18}.
Therefore, in principle, the Seiberg--Witten Floer stable homotopy type contains all Floer-theoretic information from Seiberg--Witten theory for rational homology $3$-spheres.

In this paper, we will consider a spin rational homology $3$-sphere $(Y,\frakt)$ embedded into a spin $4$-manifold $(X,\fraks)$ with the rational homology of $S^{1} \times S^{3}$ so that the fundamental class of $Y$ generates $H_{3}(X;\Z)$.
The main theorem of this paper states that, if $X$ admits a metric with positive scalar curvature (PSC), we can determine the Seiberg--Witten Floer stable homotopy type of such $(Y,\frakt)$, denoted by $\SWF(Y,\frakt)$, up to the local equivalence relation explained below.
This result gives a strong obstruction to PSC metrics of spin rational homology $S^{1} \times S^{3}$, and this is the authors' original motivation for this study.
To the same class of $4$-manifolds, there are two known obstructions based on Seiberg--Witten theory, Jianfeng Lin's obstruction~\cite{Lin19} and the authors' obstruction~\cite{KT20} explained later, and the main theorem of this paper recovers both of them.

To motivate to consider the local equivalence relation, let us recall several homology cobordism invariants from Seiberg--Witten theory.
Applying various equivariant ordinary/generalized cohomologies to the Seiberg--Witten Floer stable homotopy type, many numerical homology cobordism invariants can be extracted, such as, the Fr{\o}yshov invariant~\cite{Fr96,Fr10}, which we denote by $\delta$ following \cite{Ma16}, Manolescu's invariants $\alpha, \beta, \gamma, \kappa$ \cite{Ma14,Ma16}, and 
Stoffregen's invariants $\overline{\delta}, \underline{\delta}$~\cite{Sto171}.
These invariants have different applications, for example:
The Fr{\o}yshov invariant $\delta$ was used to extend Donaldson's diagonalization theorem~\cite{Do83} to negative-definite $4$-manifolds with boundary~\cite{Fr96,Fr10,Ma03}.
Manolescu used the invariant $\beta$ to disprove the triangulation conjecture~\cite{Ma16}, and used $\kappa$ to extend Furuta's 10/8-inequality~\cite{Fu01} to spin $4$-manifolds with boundary \cite{Ma14}.
Stoffregen's invariants $\overline{\delta}, \underline{\delta}$ should correspond, respectively, to $\overline{d}, \underline{d}$ in involutive Heegaard Floer homology~\cite{HM17}, using $\Z/4$-equivariant ordinary cohomology.


These invariants $\alpha, \beta, \gamma, \delta,\overline{\delta}, \underline{\delta}, \kappa$ are spin rational homology cobordism invariants, and obtained from $\SWF(Y,\frakt)$ described above.
However, these invariants factor through a weaker invariant than $\SWF(Y,\frakt)$, the {\it local equivalence class} of $\SWF(Y,\frakt)$, defined by Stoffregen~\cite{Sto20}.
The local equivalence is an equivalence relation on a certain class of spaces including $\SWF(Y,\frakt)$ for rational homology $3$-spheres $Y$, and this is an abstraction of a relation between $\SWF(Y_{0}, \frakt_{0})$ and $\SWF(Y_{1}, \frakt_{1})$ for $(Y_{0}, \frakt_{0})$ and $(Y_{1}, \frakt_{1})$ which are spin rational homology cobordant to each other.
To summarize this situation,
let us denote by $\Theta_{\Z}^{3}$ the $3$-dimensional homology cobordism group, and 
denote by $\Theta_{\Q, \rm spin}^{3}$ the $3$-dimensional spin rational homology cobordism group.
Namely, an element of $\Theta_{\Q, \rm spin}^{3}$ is the equivalence class $[(Y,\frakt)]$ of a spin rational homology $3$-sphere, and the equivalence relation is given by a spin rational homology cobordism.
Stoffregen~\cite{Sto20} introduced the {\it local equivalence group} $\LE$, which consists of the local equivalence classes of certain spaces modeled on $\SWF(Y,\frakt)$.
Then one has group homomorphisms
\[
\Theta_{\Z}^{3} \to \Theta_{\Q, \rm spin}^{3} \to \LE.
\]
For a spin rational homology $3$-sphere $(Y,\frakt)$, the local equivalence class $[\SWF(Y,\frakt)]$ is valued in $\LE$, and
the above numerical invariants $\alpha, \beta, \gamma, \delta,\overline{\delta}, \underline{\delta}, \kappa$ factor through $\LE$, such as $\alpha(Y,\frakt) =\alpha([(Y,\frakt)]) = \alpha([\SWF(Y,\frakt)])$:
\[
\Theta_{\Q, \rm spin}^{3} \to \LE
\xrightarrow{\alpha, \beta, \gamma, \delta,\overline{\delta}, \underline{\delta}, \kappa} \Q.
\]

\subsection{Main theorem}

As described, the local equivalence class $[\SWF(Y,\frakt)] \in \LE$ of $\SWF(Y,\frakt)$ is, so far at least,  a candidate of the `universal' Seiberg--Witten theoretic homology cobordism invariant of $(Y,\frakt)$:
it contains information of all known homology cobordism invariants obtained from Seiberg--Witten theory.
The main theorem of this paper determines $[\SWF(Y,\frakt)]$ when $Y$ is embedded into a spin rational homology $S^1\times S^3$ admitting a PSC metric so that $Y$ generates $H_3(X;\Z)$:

\begin{theo}
\label{real main}
Let $(X,\fraks)$ be an oriented spin rational homology $S^1\times S^3$, and $(Y,\frakt)$ be an oriented spin rational homology $3$-sphere.
Suppose that $(Y,\frakt)$ is a cross-section of $(X,\fraks)$, i.e. $Y$ is embedded into $X$ so that it represents a fixed generator of $H_3(X;\Z)$, and that $\fraks|_{Y}$ is isomorphic to $\frakt$.
Assume that $X$ admits a PSC metric.
Then the local equivalence class of $\SWF(Y,\frakt)$ is given by
\begin{align}
\label{eq: main theo determination}
[\SWF(Y,\frakt)] 
= \left[ \left( S^{0},0,-\frac{\lambda_{SW} (X, \s)}{2}\right) \right].
\end{align}
In particular, for an arbitrary spin rational homology cobordism invariant which factors through $\LE$, the invariant of $(Y,\frakt)$ coincides with the invariant of the right-hand side of \eqref{eq: main theo determination}.
\end{theo} 


Here $\lambda_{SW}(X, \fraks)$ is the Casson-type invariant defined by the Mrowka--Ruberman--Saveliev~\cite{MRS11} for an integral homology $S^{1} \times S^{3}$, which was later generalized for a rational homology $S^{1} \times S^{3}$ by J.~Lin--Ruberman--Saveliev~\cite{LRS17}.
Recall that an element of $\LE$ is expressed as the class of a triple $(Z,m,n)$, where $Z$ is  a space of type SWF~\cite[Definition~2.7]{Ma16}, and $m \in \Z$, $n \in \Q$.

\begin{rem}\label{ooo}In this paper, we developed Seiberg--Witten theory for 4-manifolds with periodic ends to prove \cref{real main}.
But we expect that an alternative proof of \cref{real main} without using Seiberg--Witten theory for 4-manifolds with periodic ends could be given by using Schoen--Yau's argument \cite{SY86} combined with a kind of gluing theorems for relative Bauer--Furuta invariants \cite{KLS18, KLS'18, SS21}.
\end{rem}

\subsection{Obstructions to PSC metrics}
\label{subsection: Obstructions to PSC metrics}

Now we regard \cref{real main} as an obstruction to PSC metrics on homology $S^{1} \times S^{3}$, and compare this with known obstructions on PSC metrics for the same class of $4$-manifolds.
We can extract from \cref{real main} convenient obstructions to PSC metrics, and moreover that \cref{real main} provides a systematic way to recover prior results.

Recall that it is well-understood which rational homology $3$-spheres admit PSC metrics: only connected sums of spherical $3$-manifolds.
Rational homology $S^{1} \times S^{3}$ is a class of $4$-manifold that may be seen to be closed to rational homology $3$-sphere, but it is not easy to rule out the existence of PSC metrics on such $4$-manifolds.
In dimension $4$, the Seiberg--Witten invariant is known as a powerful obstruction to PSC metric, 
but it cannot be used to rational homology $S^{1} \times S^{3}$, since the Seiberg--Witten invariant is not well-defined for such $4$-manifolds.
J.~Lin recently made a breakthrough in this situation:
he gave the first obstruction to PSC metric based on Seiberg--Witten theory for integral homology $S^{1} \times S^{3}$ in \cite{Lin19}, and later this result was generalized by himself with Ruberman and Saveliev to any rational homology $S^{1} \times S^{3}$ in \cite{LRS17}.
J.~Lin's obstruction is described as follows: under the same assumption with \cref{real main}, one has the equality
\begin{align}
\label{eq: Lin formula}
\delta(Y, \frakt) = \lambda_{SW}(X, \fraks).
\end{align}
Using \cref{real main}, we can give an alternative proof of J.~Lin's formula \eqref{eq: Lin formula}, and further generalize it to various Fr{\o}yshov-type invariants:

\begin{cor}
\label{main cor}
Let $(X,\fraks)$ be an oriented spin rational homology $S^1\times S^3$, and $(Y,\frakt)$ be an oriented spin rational homology $3$-sphere.
Suppose that $(Y,\frakt)$ is a cross-section of $(X,\fraks)$.
Assume that $X$ admits a PSC metric.
Then we have
\begin{align}
\label{eq: 6 eq}
\alpha(Y,\frakt) = \beta(Y,\frakt) = \gamma(Y,\frakt) = \delta(Y,\frakt) = \overline{\delta}(Y,\frakt)= \underline{\delta}(Y,\frakt) = \kappa(Y,\frakt) = \lambda_{SW} (X, \s).
\end{align}
\end{cor}

\begin{proof}
By the definition of $\alpha, \beta, \gamma, \delta,\overline{\delta}, \underline{\delta}, \kappa$ \cite{Ma16,Ma14,Sto171}, it is easy to see that the values of these invariants for the right-hand side of \eqref{eq: main theo determination} are given by $\lambda_{SW}(X,\fraks)$.
Therefore the \lcnamecref{main cor} directly follows from \cref{real main}.
\end{proof}

Note that, by \cref{main cor}, we can replace $\lambda_{SW}(X,\fraks)$ in the right-hand side of \eqref{eq: main theo determination} with various invariants of $(Y,\frakt)$.

An obvious consequence of \cref{main cor} is:

\begin{cor}
\label{cor: glued}
Let $Y$ be an oriented homology 3-sphere. 
Suppose that at least two of $\alpha(Y), \beta(Y), \gamma(Y), \delta(Y),\overline{\delta}(Y), \underline{\delta}(Y), \kappa(Y)$ do not coincide with each other.
Then, for any homology cobordism $W$ from $Y$ to itself, the homology $S^1\times S^3$ obtained from $W$ by gluing the boundary components does not admit a PSC metric. 
\end{cor}

Here we drop the unique spin structure from our notation for (integral) homology $3$-spheres.

J.~Lin~\cite{Lin19} and J.~Lin--Ruberman--Saveliev~\cite{LRS17} used monopole Floer homology to establish the obstruction \eqref{eq: Lin formula}.
Morally, our argument in this paper can be thought of as a stable cohomotopy version of J.~Lin's argument in \cite{Lin19}.


After J.~Lin's work, the authors~\cite{KT20} gave another obstruction based on a 10/8-type inequality, described in \cref{theo: first KT}.
Using \cref{main cor} combined with Manolescu's relative 10/8-inequality~\cite{Ma14},
we can give an alternative proof of the authors' previous result (with a minor change):

\begin{cor}[\cite{KT20}]
\label{theo: first KT}
Let $(X,\fraks), (Y,\frakt)$ be as in \cref{real main}.
Take a compact smooth spin $4$-manifold $M$ bounded by $(Y,\frakt)$.
Suppose that $(Y,\frakt)$ is a cross-section of $(X,\fraks)$.
Assume that $X$ admits a PSC metric.
Then we have
\begin{align}
\label{eq first KT}
b^{+}(M) \geq -\frac{\sigma(M)}{8} - \delta(Y,\frakt)-1.
\end{align}
\end{cor}

\begin{proof}
Manolescu's relative 10/8-inequality, which is \cite[Theorem~1]{Ma14} generalized to a rational homology $3$-sphere (see \cite[Remark~2]{Ma14}), implies that
\[
b^{+}(M) \geq -\frac{\sigma(M)}{8} - \kappa(Y,\frakt) -1.
\]
Combining this with \eqref{eq: 6 eq}, we obtain \eqref{eq first KT}.
\end{proof}

\begin{rem}
The inequality \eqref{eq first KT} is slightly weaker than the original inequality given in \cite[Theorem~1.1]{KT20}.
The source of this difference is that, in \cite{KT20}, we used Furuta-Kametani's 10/8-type inequality~\cite{FK05} based on $KO$-theory, whereas Manolescu's inequality is based on $K$-theory.
\end{rem}

\subsection{Outline of the proof of the main theorem}
\label{subsec: Outline of the proof of the main theorem}
Here is an explanation of an outline of the proof of \cref{real main}.
The heart of this paper is, under the assumption of the existence of PSC metric on $X$, to consider finite-dimensional approximations of the Seiberg--Witten equations on a periodic-end $4$-manifold.
More precisely, we shall construct a relative Bauer--Furuta-type invariant over a half-periodic-end $4$-manifold
\[
W[-\infty,0] =  \cdots  \cup_{Y} W\cup_{Y} W \cup_{Y} W,
\]
along the spirit of Furuta~\cite{Fu01}, Bauer--Furuta~\cite{BF04}, and Manolescu~\cite{Ma03}. 
Here $W$ is the $4$-manifold defined by cutting $X$ open along $Y$, and the `left side' end is equipped with a periodic PSC metric and a neighborhood of the `right side' boundary is equipped with a product metric of the form $[0,1] \times Y$. Technically, the relative Bauer--Furuta invariant over such a non-compact 4-manifold is defined using the similar method given in \cite{IT20} which defines the relative Bauer--Furuta invariant for a certain class of 4-manifolds with conical end.

The key observation is that $W[-\infty,0]$ with such a periodic PSC metric on the end looks like a homology cobordism from $S^{3}$ to $Y$ from Seiberg--Witten theoretic point of view.
The relative Bauer--Furuta invariant over $W[-\infty,0]$ gives a local map from $\left[ \left( S^{0},0,-\lambda_{SW} (X, \s)/2\right) \right]$ to $\SWF(Y,\frakt)$.
The quantity $\lambda_{SW} (X, \s)$ emerges from the spin Dirac index over $W[-\infty,0]$, discussed in \cref{Dirac index on Winf}.

Similarly, by considering the relative Bauer--Furuta invariant over
\[
W[0, \infty] = W \cup_{Y} W\cup_{Y} W \cup_{Y} \cdots,
\]
we get a local map from $\SWF(Y,\frakt)$ to $\left[ \left( S^{0},0,-\lambda_{SW} (X, \s)/2\right) \right]$, and we can conclude that $\SWF(Y,\frakt)$ is locally equivalent to $\left[ \left( S^{0},0,-\lambda_{SW} (X, \s)/2\right) \right]$.

\subsection{Examples}

In \cref{section: Examples} we shall give examples of concrete 3-manifolds $Y$ to which we can apply the obstructions given in \cref{subsection: Obstructions to PSC metrics}.
Here let us exhibit a part of those examples.

As a consequence of his formula \eqref{eq: Lin formula},
J.~Lin proved in \cite[Corollary 1.3]{Lin19} that a homology $S^1\times S^3$ having a cross-section $Y$ with $\mu(Y) \neq \delta(Y) \mod 2$ does not admit a PSC metric.
Here $\mu(Y) \in \Z/2\Z$ denotes the Rohlin invariant.
For Seifert homology 3-spheres, we can get an `integer-valued lift' of this result by J.~Lin.
Moreover, also for linear combinations of Seifert homology 3-spheres of certain type, we can get an obstructions described in terms of some integer-valued invariants of certain 3-manifolds:

\begin{theo}The following statements hold: 
\begin{itemize}
\item[(i)]
Let $Y'$ be a Seifert homology 3-sphere such that
\[
 -\overline{\mu} (Y') \neq \delta(Y' ) , 
 \]
 where $\overline{\mu}$ is the Neumann--Siebenmann invariant for graph homology 3-spheres, introduced in \cite{N80,Si80}. 
 Let $Y$ be an oriented homology 3-sphere which is homology cobordant to $Y'$.
 Then, for any homology cobordism $W$ from $Y$ to itself, the 4-manifold obtained from $W$ by gluing the boundary components does not admit a PSC metric.   
 \item[(ii)] Let $Y_1, \cdots, Y_n$ be negative Seifert homology 3-spheres of projective type. Suppose that $\delta(Y_1) \leq \cdots \leq \delta(Y_n)$. Set $\wt{\delta}_i := \delta(Y_i) + \overline{\mu} (Y_i)$. Suppose that at least two of following four integers are distinct:
 \begin{align*}
 \sum_{i=1}^n {\delta} ( Y_i), \quad
 2 \lfloor \frac{\sum_{i=1}^n \wt{\delta}_i +1}{2}\rfloor -\sum_{i=1}^n \overline{\mu} ( Y_i),\\
 2\lfloor \frac{\sum_{i=1}^{n-1} \wt{\delta}_i +1}{2} \rfloor  -\sum_{i=1}^n \overline{\mu} ( Y_i), \quad 
 2\lfloor \frac{\sum_{i=1}^{n-2} \wt{\delta}_i +1}{2} \rfloor    -\sum_{i=1}^n \overline{\mu} ( Y_i).
  \end{align*}
Let $Y$ be an oriented homology 3-sphere which is homology cobordant to $Y_1 \# \cdots \# Y_n$. 
Then, for any homology cobordism $W$ from $Y$ to itself, the 4-manifold obtained from $W$ by gluing the boundary components  does not admit a PSC metric. 
 \end{itemize}
\end{theo} 
For the definition of projective Seifert homology 3-spheres, see \cref{section: Examples}.

\subsection{Outline of this paper}
We finish off this introduction with an outline of the contents of this paper.
The contents until \cref{section: Relative Bauer--Furuta type invariant} are devoted to construct the relative Bauer--Furuta invariant on the periodic-end 4-manifold $W[-\infty,0]$.
In \cref{section: Preliminaries} we give several notations related to infinite cyclic covering spaces of a $4$-manifold. We also review Fredholm theory for infinite cyclic covering spaces, Seiberg--Witten Floer homotopy types and notion of local equivalence. In \cref{section: Fredholm theories} we ensure Fredholm properties of elliptic operators on certain 4-manifolds with periodic end and boundary. 
We calculate cohomologies of the Atiyah--Hitchin--Singer operator on such non-compact 4-manifolds.
We also calculate the Dirac index on $W[-\infty,0]$ in \cref{Dirac index on Winf}.
In \cref{section: The boundedness result} we show a boundedness result which is needed to construct the relative Bauer--Furuta invariant. In \cref{section: Relative Bauer--Furuta type invariant} we construct the relative Bauer--Furuta invariant for the 4-manifolds $W[-\infty,0]$ with periodic end and boundary.
In \cref{section: The proof of main} we prove \cref{real main} along the idea explained in \cref{subsec: Outline of the proof of the main theorem}.
In \cref{section: Obstruction to embedding of 3-manifolds into 4-manifolds with PSC metric} we give a generalization of \cref{real main}, which is stated as an obstruction of embeddings of $3$-manifolds into $4$-manifolds admitting PSC metrics.
In \cref{section: Examples} we provide several families of examples of homology $S^1\times S^3$'s which cannot admit PSC metrics using \cref{real main}.

\begin{acknowledgment}
The authors would like to express their gratitude to the organizers and participants of Gauge Theory Virtual
for giving them an opportunity to reconsider their past work \cite{KT20}. The authors also wish to thank Nobuo Iida for discussing \cref{ooo} with us. 
The first author was partially supported by JSPS KAKENHI Grant Numbers 17H06461, 19K23412, and 21K13785.
The second author was supported by JSPS KAKENHI Grant Number 20K22319 and RIKEN iTHEMS Program.
\end{acknowledgment}

\section{Preliminaries} \label{section: Preliminaries}

\subsection{Notations}
\label{subsec: Notations}

In this \lcnamecref{subsec: Notations} we introduce several notations on periodic $4$-manifolds. 
Let $(X,\fraks)$ be an oriented spin rational homology $S^1 \times S^3$, i.e. a spin $4$-manifold whose rational homology is isomorphic to that of $S^1 \times S^3$.
Fix a Riemannian metric $g_X$ on $X$ and a generator of $H_3(X;\Z)$, denoted by $1\in H_3(X;\Z)$.
Note that $H_3(X;\Z)$ is isomorphic to $H^{1}(X;\Z)$, and hence to $\Z$.
Let $Y$ be an oriented rational homology $3$-sphere, and
assume that $Y$ is embedded into $X$ so that $[Y]=1$.
We call such $Y$ a {\it cross-section} of $X$.
Let $W_0$ be the rational homology cobordism from $Y$ to itself obtained by cutting $X$ open along $Y$.
The manifold $W_{0}$ is equipped with an orientation and a spin structure induced by those of $X$.
For $(m,n) \in (\{-\infty \}\cup \Z) \times (\Z\cup \{\infty\})$ with $m<n$,
we define the periodic $4$-manifold
\[
W[m,n]:= W_m \cup_Y W_{m+1} \cup_Y \dots \cup_YW_n,
\]
where $W_i$ is a copy of $W_0$ for each $i \in \Z$.
This $4$-manifold $W[m,n]$ is also equipped with an orientation and a spin structure as well as $W_{0}$.
The element of $H^1(X;\Z)$ corresponding to $1 \in H_3(X;\Z)$ via the Poincar\'e duality gives the isomorphism class of an infinite cyclic covering
\begin{align}\label{eq: zcov}
p:\widetilde{X} \to X
\end{align}
and an identification
\begin{align}\label{eq: zcov1}
\widetilde{X} \cong W[-\infty,\infty].
\end{align}

Via the identification \eqref{eq: zcov1}, let us think of $p$ as a map from $W[-\infty ,\infty]$ to $X$. Define the map $p_-:W[-\infty ,0] \to X$ as the restriction of $p$.
 We call an object defined on $W[-\infty ,0]$, such as connection, metric, bundle, and differential operator, a {\it periodic object} if the restriction of the object to $W[-\infty, 0]$ can be identified with the pull-back of an object on $X$ under $p_-$.
Considering the pull-back under $p_{-}$, the Riemannian metric $g_X$ on $X$ induces a Riemannian metric, denoted by $g_{W[-\infty ,0] }$, on $W[-\infty ,0] $.
Let $S^+, S^-$ be the positive/negative spinor bundles respectively over $W[-\infty ,0]$ with respect to the metric and the spin structure above.
Fixing a trivialization of the determinant line bundle of the spin structure on $W[-\infty ,0]$, we obtain the canonical reference connection $A_0$ on $W[-\infty ,0]$ corresponding to the trivial connection.
 
 To consider the weighted Sobolev norms on $W[-\infty, 0]$, fix a function  
\[
\tau : \wt{X} \to \R
\]
with $T^*\tau =\tau+ 1$, where $T: \wt{X} \to \wt{X}$ is the deck transform determined by $T(W_i)=W_{i-1}$. 
Note that $d\tau$ defined a cohomology class $[d\tau] \in H^1(X; \Z)$ which is equal to $1 \in H^1(X; \Z)$ corresponding to $1 \in H_3(X;\Z)$ via the Poincar\'e duality.

\begin{defi}
\label{defi: weighted}
Let $E$ be a periodic vector bundle on $W[-\infty, 0]$ with a periodic inner product.
For a fixed $k>0$ and $\delta \in \R$, we define the {\it weighted Sobolev norm} by
\[
\|f\|_{L^2_{k, \delta} (W[-\infty, 0])} := \| e^{\delta \tau} f\|_{L^2_k (W[-\infty, 0])} . 
\]
for a smooth comactly supported section $f$ of $E$.
Here we used a periodic metric and a periodic connection on $E$ to define the $L^{2}_{k}$-norm.
Let $L^2_{k, \delta} (E)$ denote the $L^2_{k, \delta}$-completion of compactly supported smooth sections of $E$.
\end{defi}

Note that the equivalence class of norms $\|-\|_{L^2_{k, \delta} (W[-\infty, 0])} $ does not depend on the choices of a periodic metric and a periodic connection on $E$.

\subsection{Fredholm theory on $\wt{X}$} \label{delta0}
\label{subsection review Fredholm}

In this \lcnamecref{subsection review Fredholm}
we review the Fredholm property of periodic elliptic operators on the infinite cyclic covering $\wt{X}$ developed by  C.~Taubes~\cite{T87}.
He showed that a periodic elliptic operator is Fredholm under some condition with respect to $L^2_{k,\delta}$-norms for generic $\delta \in \R$.
For the details, see \cite{T87}, or \cite[Subsection~2.1]{KT20}.

Let $\mathbb{D}= (D_i,E_i)$ be a periodic elliptic complex on $\wt{X}$, i.e. the complex 
 \begin{align} \label{comp}
0\to \Gamma(\wt{X};E_N) \xrightarrow{D_N} \Gamma(\wt{X};E_{N-1} )  \to \cdots \xrightarrow{D_1} \Gamma(\wt{X};E_0)\to 0 
\end{align}
consisting of first order periodic linear differential operators $D_i$ between periodic vector bundles $E_{i}$ on $\wt{X}$ with exact symbol sequence.
Here, for a vector bundle $E$, the notation $\Gamma (  \wt{X}, E)$ denotes the set of compactly supported smooth sections of $E$.
As well as \cref{defi: weighted},
define the weighted Sobolev norm on $\wt{X}$ by
 \[
 \| f \|_{L^2_{k,\delta}(\wt{X})} := \| e^{\tau \delta} f \|_{L^2_{k}(\wt{X})} 
 \]
using a periodic connection and a periodic metric. 
The complex \eqref{comp} gives rise to the complex of bounded operators 
 \begin{align}\label{elliop}
L^2_{k+N+1,\delta} (\wt{X};E_N) \xrightarrow{D_N} L^2_{k+N,\delta}  (\wt{X};E_{N-1} )  \to \cdots \xrightarrow{D_1} L^2_{k,\delta} (\wt{X};E_0) 
\end{align}
for each $k>0$ and $\delta \in \R$.

Note that, since the operators in \eqref{elliop} are periodic differential operators, there exist differential operators $\hat{\mathbb{D}}=(\hat{D}_i, \hat{E}_i)_{i=0 , \cdots , N} $ on $X$ such that $\mathbb{D}$ is given as the pull-back $p_-^*\hat{\mathbb{D}}$.

\begin{defi}
For $z \in \C$, define the complex $\hat{\mathbb{D}}(z)$ by 
\[
0 \to \Gamma(X;\hat{E}_N) \xrightarrow{\hat{D}_N(z)} \Gamma (X;\hat{E}_{N-1} )  \to \cdots \xrightarrow{\hat{D}_1(z)} \Gamma (X;\hat{E}_0) \to 0, 
\]
where the operators $\hat{D}_i(z): \Gamma(X;\hat{E}_i) \to \Gamma (X;\hat{E}_{i-1} )$ are defined by
\[
\hat{D}_i(z)(f):= e^{-\tau z} \hat{D}_i (e^{\tau z}f).
\]
\end{defi} 

 \begin{theo}[{\cite[Lemmas 4.3 and 4.5]{T87}}]\label{fred}
 Suppose that there exists $z_0 \in \C$ where the complex $\hat{\mathbb{D}}(z_0)$ is acyclic.
Then there exists a discrete subset $\mathcal{D}$ in $\R$ with no accumulation points such that the complex \eqref{elliop} is an acyclic complex for all $\delta$ in $\R \setminus \mathcal{D}$.
 \end{theo}
 

\begin{defi}[\cite{RS07}]  We call $g_X$ an {\it admissible metric} on $X$ if the kernel of 
\[
D^+_{A_0} + f^* d\theta : L^2_k (X;S^+) \to   L^2_{k-1} (X;S^-)
\]
is zero, where the map $f:X\to S^1$ is a smooth classifying map of \eqref{eq: zcov}. 
\end{defi}
The admissibility condition does not depend on the choice of classifying map $f$.
One can show that every PSC metric on $X$ is an admissible metric (See (2) in \cite{RS07}). 

In \cite{KT20}, we confirmed that \cref{fred} can be used for differential operators appearing as the linearization of the Seiberg--Witten equations:
\begin{lem} [{\cite[Lemma~2.6]{KT20}}] \label{allops}
The assumption of \cref{fred} is satisfied for the following operator/complexes: 
\begin{itemize}
\item The Dirac operator $D^+_{A_0}: L^{2}_{k,\delta} (\wt{X};S^+) \to L^{2}_{k-1,\delta} (\wt{X};S^-)$ with respect to the pull-back of an admissible metric $g_X$ on $X$.
\item The Atiyah--Hitchin--Singer complex
\[
0 \to L^{2}_{k+1,\delta} (i\Lambda^0(\wt{X})) \xrightarrow{d}  L^{2}_{k,\delta} (i\Lambda^1(\wt{X})) \xrightarrow{d^+} L^{2}_{k-1,\delta} (i\Lambda^+(\wt{X})) \to 0.
\]
\item The de Rham complex 
\[
0 \to L^{2}_{k+1,\delta} (i\Lambda^0(\wt{X})) \xrightarrow{d}  L^{2}_{k,\delta} (i\Lambda^1(\wt{X})) \xrightarrow{d} \cdots \xrightarrow{d}  L^{2}_{k-3,\delta} (i\Lambda^4(\wt{X})) \to 0.
\]
\end{itemize}
\end{lem}

\begin{rem}
\label{rem deltazero}
Since the subset $\mathcal{D}$ of $\R$ given in \cref{fred} has no accumulation points, we can take a sufficiently small $\delta_0>0$ so that for any $\delta \in (0,\delta_0)$ the operators in \cref{allops} are Fredholm. 
Henceforth we fix the notation $\delta_0$.
\end{rem}

\subsection{Seiberg--Witten Floer stable homotopy type} 
\label{subsection: Seiberg--Witten Floer stable homotopy type}
In the proof of \cref{real main}, we use a variant of the relative Bauer--Furuta invariant for 4-manifolds with periodic end. 
In this subsection we review several notions of Manolescu's Seiberg--Witten Floer stable homotopy type, which is necessary to describe the relative Bauer--Furuta invariant.
The main references of this subsection are Manolescu~\cite{Ma03} and Khandhawit~\cite{Kha15}. 

Let $Y$ be an oriented rational homology $3$-sphere
with a Riemannian metric $g_Y$. 
Let $\frakt$ be a spin$^c$ structure on $Y$, and $S$ be the spinor bundle of $\frakt$. 
We fix a flat spin$^{c}$ reference connection $a_0$ of the determinant line bundle of $S$.

\begin{defi}
For an integer $k>2$, we define the {\it configuration space} by
\[
\mathcal{C}_k(Y,\frakt):=(a_0+  L^2_{k-\frac{1}{2}}  (i \Lambda^1_Y))\oplus  L^2_{k-\frac{1}{2}}  ( S ).
\]
The {\it Chern--Simons--Dirac functional  } 
$CSD : \mathcal{C}_k(Y,\frakt) \to \R$ is deined
 by 
\[
CSD (a,\phi) :=  \frac{1}{2}   \left(-\int_Y a\wedge da + \int_Y \left<\phi , \D^+_ {a_0+a} \phi \right>\t{dvol}_Y  \right), 
\]
where $\D^+_ {a_0+a}$ is the $\text{spin}^c$ Dirac operator with respect to the connection $a_0+a$.

The {\it gauge group} $\G_k(Y)$ and a subgroup $\widetilde{\G}_k(Y)$ of $\G_k(Y)$ are defined by
\[
\G_k(Y) := L^2_{k+\frac{1}{2}} (Y, S^1)
\]
and
\[
\widetilde{\G}_k(Y) := \Set{ g \in \G_k(Y) | g= e^{if},\ \int_{Y} f \text{vol}_Y =0 }  .
\]
\end{defi}

The gauge group $\G_k(Y)$ naturally acts on $\mathcal{C}_k(Y,\frakt)$ and the functional $CSD$ is invariant under the action.
The global slice of the action of $\widetilde{\G}_{k}(Y) $ on $\mathcal{C}_k(Y,\frakt)$ is given by 
 \[
  V_k(Y,\fraks)= (\Ker d^*:  L^2_{k-\frac{1}{2}}(\Lambda_Y^1) \to L^2_{k-\frac{3}{2}}(\Lambda_Y^0))  \oplus  L^2_{k-\frac{1}{2}}  ( S ),
  \]
 on which we still have the remaining $S^{1}$-action.
 We often drop $k$ and/or $(Y,\fraks)$ from our notation to denote $V_k(Y,\fraks)$.
 The $S^1$-equivariant formal gradient flow on $V(Y,\fraks)$ of $CSD$ with respect to the Coulomb projection of the $L^2$-metric can be written as the sum of the linear term
 \[
 l= (*d,  \D_{a_0}) : V_k(Y,\fraks) \to V_{k-1}(Y,\fraks)
 \]
  and some quadratic term, denoted by $c : V_k(Y,\fraks) \to V_{k-1}(Y,\fraks)$.

For $\lambda < 0 < \mu$, we define $V_\lambda^\mu(Y)$ as the direct sum of eigenspaces of $l$, regarded as an unbounded operator on $V_{1/2}(Y,\fraks)$, whose eigenvalues belong to $(\lambda,\mu]$. Here we think of $V_\lambda^\mu(Y)$ as a subspace of $V_k(Y,\fraks)$. We denote by
\[
 p_\lambda^\mu: V_k(Y,\fraks) \to V_\lambda^\mu(Y)
\]
 the $L^2$-projection of $V_k(Y,\fraks)$ onto $V_\lambda^\mu(Y)$.
We often abbreviate $V_\lambda^\mu(Y)$ as $V_\lambda^\mu$.
 Since $l$ is the sum of a real operator and a complex operator, $ V_\lambda^\mu$ decomposes into a real vector space and a complex vector space, denoted by
  \begin{align}\label{de}
 V_\lambda^\mu = V_\lambda^\mu(\R) \oplus V_\lambda^\mu(\mathbb{C}).
   \end{align}

Let us use basic terms of Conley index theory following \cite[Section~5]{Ma03}.
 Manolescu proved some compactness result~\cite[Proposition~3]{Ma03}, and as a consequence, it turns out that a closed ball in $V_\lambda^\mu$ of sufficiently large radius centered at the origin is an isolating neighborhood of the invariant part of the ball.
Precisely, the flow on  $V_\lambda^\mu$ considered here is a flow obtained from $(l + p_\lambda^\mu c)$ by cutting off outside a larger ball (see \cite[page~907]{Ma03}).
 We denote by $I_\lambda^\mu$ the $S^1$-equivariant Conley index of  the invariant part.
The Seiberg--Witten Floer homotopy type $\SWF(Y, \frakt)$ is defined as the triple 
$(\Sigma^{-V^0_\lambda} I_\lambda^\mu,0,n(Y, \frakt, g_Y))$, which is symbolically denoted by
\[
 \SWF(Y, \frakt) = \Sigma^{-n(Y, \frakt, g_Y) \mathbb{C} -V^0_\lambda} I_\lambda^\mu.
\]
The triple is regarded as an object a certain suspension category $\mathfrak{C}$.
In general an object of $\mathfrak{C}$ is given as a triple $(Z,m,n)$, where $Z$ is a pointed topological $S^{1}$-space, $m \in\Z$, and $n \in \Q$.
The quantity $n(Y, \frakt, g_Y) \in \Q$ is defined to be 
\begin{align} \label{n}
n(Y, \frakt, g_Y):= \ind_\C D^+ + \frac{\sigma(W)}{8},  
\end{align}
where $(W, \frakt')$ is a compact $\spinc$ 4-manifold satisfying $\partial (W, \frakt')= (Y, \frakt)$ and $\ind_\C D^+$ means the index of the Dirac operator with APS boundary condition. 
For the meaning of formal desuspensions, see \cite{Ma03}.

Here let us consider the case when the spin$^{c}$ structure $\frakt$ comes from a spin structure.
In this case, the formal gradient flow of $CSD$ admits a larger symmetry of the group $\Pin(2) $ defined by 
  \[
  \Pin(2) := S^1 \cup j S^1 \subset  Sp(1).
  \]
This group  $\Pin(2)$ acts on $V_k(Y,\fraks)$ for any non-negative integer $k$ as follows:
the $\Pin(2)$-action on spinors given as the restriction of the natural $Sp(1)$-action on spinor bundles, and the $\Pin(2)$-action on $\Om^1_Y$ is given via the non-trivial homomorphism $\Pin(2)\to O(1)$.
We denote by $\wt{\R}$ the real $1$-dimensional representation of $\Pin(2)$, and by $\quat$ the space of quaternions, on which $\Pin(2)$ naturally acts. Thus we have decompositions  
\[
V_k(Y,\fraks) = V({\R}) \oplus V(\quat)
\]
and 
\begin{align}\label{decom}
 V_\lambda^\mu = V_\lambda^\mu(\R) \oplus V_\lambda^\mu(\quat). 
\end{align}

Considering $\Pin(2)$-equivariant Conley index instead, we obtain a stable homotopy type of a pointed $\Pin(2)$-space 
\[
 \SWF(Y, \frakt) =  \Sigma^{-\frac{n(Y, \frakt, g)}{2} \mathbb{H} -V^0_\lambda}  I^\mu_\lambda,
 \]
which lies in a suspension category $\mathfrak{C}'$.
An object of $\mathfrak{C}'$ is given as a triple $(Z,m,n)$, where $Z$ is a pointed topological $\Pin(2)$-space, $m \in\Z$, and $n \in \Q$.

Let us recall the definition of local equivalence. 
\begin{defi}[\cite{Sto20}]
For two objects $(Z_1, m_1, n_1)$ and $(Z_2, m_2, n_2)$ in $\mathfrak{C}'$, a {\it local map} is a $\Pin(2)$-equivariant map 
\[
f: \Sigma^{ (N-n_1) \quat } \Sigma^{ (M-m_1) \wt{\R}  }Z_1 \to \Sigma^{ (N-n_2) \quat } \Sigma^{ (M-m_2) \wt{\R}  } Z_2
\]
for some $M \in \Z$ and $N \in \Q$ such that the $S^1$-invariant part $f^{S^1}$ is a $\Pin(2)$-homotopy equivalence.
Two objects $(Z_1, m_1, n_1)$ and $(Z_2, m_2, n_2)$ are {\it locally equivalent} if there exist local maps $f : (Z_1, m_1, n_1) \to (Z_2, m_2, n_2)$ and $g : (Z_2, m_2, n_2) \to (Z_1, m_1, n_1)$. 

\end{defi}
Typical examples of local maps are obtained as the relative Bauer--Furuta invariants for negative definite spin cobordisms between rational homology 3-spheres.

\subsection{The Seiberg--Witten equations on $W[-\infty, 0]$} 
In this subsection we describe the Seiberg--Witten equations on $W[-\infty, 0]$, mainly to fix notations.
We use the double Cloumb gauge condition introduced in \cite{Kha15}.

\begin{defi}Let $k$ be a positive integer with $k \geq 4$ and $\delta$ a positive real number.
We first define the {\it configuration space } $\mathcal{C}_{k, \delta} (W[-\infty, 0] )$ by
\[
\mathcal{C}_{k, \delta} (W[-\infty, 0] ) 
:= (A_0, 0 ) +  L^2_{k, \delta}( i\Lambda_{W[-\infty, 0] }^1 ) \oplus L^2_{k, \delta} ( S^+_{W[-\infty, 0] }).  
\]
The gauge group 
$\G_{k+1, \delta}(W[-\infty, 0] )$
is given by 
\begin{align}\label{gauge}
\G_{k+1, \delta}(W[-\infty, 0] ) := \left\{ u : W[-\infty, 0]  \to \C ~ \middle|~ |u(x)|= 1 ~ (\forall x \in W[-\infty, 0]),~  1-u \in L^2_{k+1, \delta} (\underline{\C} ) \right\}. 
\end{align}
Here $\underline{\C}$ denotes the trivial bundle over $W[-\infty, 0]$ with fiber $\C$.
The action of $\G_{k+1, \delta}(W[-\infty, 0] )$ on $\mathcal{C}_{k, \delta} (W[-\infty, 0] )$ is given by 
\[
u \cdot (A, \Phi) := (A- u^{-1} du , u \Phi). 
\]
The {\it double Coulomb slice} introduced in \cite{Kha15} is defined by
\[
\mathcal{U}_{k, \delta}(W[-\infty, 0] )  :=    L^2_{k, \delta}( i\Lambda_{W[-\infty, 0] }^1 )_{CC} \oplus L^2_{k, \delta} ( S^+_{W[-\infty, 0] }) , 
\]
where 
\[
L^2_{k, \delta}( i\Lambda_{W[-\infty, 0] }^1 )_{CC} := \Set{ a  \in L^2_{k, \delta}( i\Lambda_{W[-\infty, 0] }^1) | d^{*_{\delta}}  a=0, d^{*} {\bf t}a=0 }.
\]
Here ${\bf t}$ denotes the restriction of $1$-forms as differential forms and $d^{*_{\delta}}$ is the formal adjoint of $d$ with respect to $L^2_{\delta}$.
\end{defi}
We will prove that $\mathcal{U}_{k, \delta}(W[-\infty, 0] )$ gives a global slice with respect to the action of $\G_{k+1, \delta}(W[-\infty, 0] )$ on $\mathcal{C}_{k, \delta} (W[-\infty, 0] ) $. Note that, on $\mathcal{C}_{k, \delta} (W[-\infty, 0] )$, the `full gauge group' 
\[
 \left\{ u : W[-\infty, 0]  \to \C ~ \middle|~ |u(x)|= 1 ~ (\forall x \in W[-\infty, 0]),~   du \in L^2_{k, \delta} (\underline{\C} ) \right\}
 \]
 also acts.
 Thus we have an additional $S^1$-symmetry on $\mathcal{U}_{k, \delta}(W[-\infty, 0] )$ coming from the limits with respect to the end of gauge transformations. 
 
Based on the Sobolev embedding  $\G_{k+1, \delta}(W[-\infty, 0] ) \to C^0(W[-\infty, 0]  , S^1) $, we can naturally define the group structure on $\G_{k+1, \delta}(W[-\infty, 0] )$ by pointwise multiplication. 

On $W[-\infty, 0] $, one can define the {\it Seiberg--Witten map}

\begin{align}
\mathcal{F}_{W[-\infty, 0] }  :\mathcal{C}_{k, \delta} (W[-\infty, 0] )  \to  L^2_{k-1, \delta}( i\Lambda_{W[-\infty, 0] }^+ \oplus S^-_{W[-\infty, 0] }) 
 \end{align}
by 
\begin{align}\label{SW}
\mathcal{F}_{W[-\infty, 0] } (A,  \Phi ) := \left( \frac{1}{2} F^+_{A^t}-\rho^{-1} ( \Phi \Phi^*)_0 , D^+_A \Phi  \right)  .
\end{align}
When we write $(a, \phi) = (A, \Phi) - (A_0, 0)$, we often decompose the Seiberg--Witten map $\mathcal{F}_{W[-\infty, 0] }$ as the sum of the linear part
\begin{align}\label{L}
L_{W[-\infty, 0] }(a, \phi) := \left( d^+a   , D^+_{A_0} \phi  \right),
\end{align}
the quadratic part
\[
C_{W[-\infty, 0] }(a, \phi) := (-(\phi \phi^*)_0,  \rho (a)\phi  ).
\]
We regard $L_{W[-\infty, 0] } $ also as an operator with domain $\mathcal{U}_{k, \delta}(W[-\infty, 0])$ by the restriction. 
The quadratic part is a compact operator by \cite[Proposition 2.13]{Lin19} for a positive $\delta$. 
The differential equation
\begin{align}\label{SW eq}
\mathcal{F}_{W[-\infty, 0] } (A,  \Phi )=0 
\end{align}
is called the {\it Seiberg--Witten equation} for $W[-\infty, 0] $. The linearlization of $\mathcal{F}_{W[-\infty, 0] }$ is given by $
L_{W[-\infty, 0] }$.

\section{Linear analysis on $W[-\infty, 0] $}  \label{section: Fredholm theories}

Fix a Riemann metric $g_{W[-\infty, 0] }$ on $W[-\infty, 0]$ such that 
\begin{itemize} 
\item $g_{W[-\infty, 0] } |_{W[-\infty ,-1]}$ is periodic and PSC, and 
\item $ g_{W[-\infty, 0] }$ is product metric near $\partial  W[-\infty, 0]  = Y$. 
\end{itemize}

\subsection{Fredholm theory on $W[-\infty, 0]$ }
In this subsection, we prove certain Fredholm properties which will be used in the proof of \cref{real main}. 
For a fixed periodic spin structure on $W[-\infty, 0]$, the spinor bundles are written as $S^+$ and $S^-$. 
In this section, we use the following completions: 
\[
L^2_{k, \delta} (i\Lambda^1_{W[-\infty, 0] }), \ L^2_{k, \delta} (i\Lambda^+_{W[-\infty, 0] }), \text{ and } L^2_{k, \delta} (S^\pm ). 
\]

We prove the Fredholm properties of the following two types of operators on $W[-\infty, 0]$: 
\begin{itemize}
\item the Atiyah--Hitchin--Singer operator with APS-boundary condition: 
\begin{align}\label{AHS}
\begin{split}
d^{*_\delta} + d^+  +  \widehat{p}^0_{-\infty} \circ  \widehat{r} :  L^2_{k, \delta}(i \Lambda_{W[-\infty, 0] }^1)  \\ 
 \to  L^2_{k-1,  \delta}(i\Lambda_{W[-\infty, 0] }^0   \oplus \Lambda_{W[-\infty, 0] }^+  )\oplus \widehat{V}^0_{-\infty}  (Y; \R), 
 \end{split}
 \end{align}
 where 
 \begin{itemize}
 \item[(i)] the space $\widehat{V}^0_{-\infty}(Y; \R) $ is the $L^2_{k-\frac{1}{2}}$-completion of the negative eigenspaces of the operator 
 \[
\widehat{l}:= \begin{pmatrix}
0 & -d^{*}   \\
-d & * d  \\ 
\end{pmatrix} : \Om^0_{Y} \oplus \Om^1_{Y}  \to \Om^0_{Y} \oplus \Om^1_{Y}, 
\]
 \item[(ii)] the map
 $\widehat{r} : L^2_{k, \delta}(i \Lambda_{W[-\infty, 0]}^1 )  \to L^2_{k-\frac{1}{2}}(\Lambda^0_{Y} \oplus \Lambda^1_{Y})$
 is the restriction, 
 \item [(iii)]the operator  
 \[
 \widehat{p}^0_{-\infty} : L^2_{k-\frac{1}{2}}(\Lambda^0_{Y} \oplus \Lambda^1_{Y}) \to \widehat{V}^0_{-\infty}  (Y; \R)
 \]
 is the $L^2$-projection to $\widehat{V}^0_{-\infty}  (Y;\R)$. 
 
 \end{itemize}

\item the Dirac operator with APS-boundary condition: 
\begin{align}\label{Dirac}
D^+_{A_0}  +  \widehat{p}^0_{-\infty} \circ  \widehat{r} :  L^2_{k,  \delta}( S^+_{W[-\infty, 0]}) 
 \to  L^2_{k-1,  \delta}(S^-_{W[-\infty, 0]})\oplus \widehat{V}^0_{-\infty}  (Y, \C), 
 \end{align}
 where 
 \begin{itemize}
 \item[(i)] the space $\widehat{V}^0_{-\infty}  (Y, \C)$ is the $L^2_{k-\frac{1}{2}}$-completion of the negative eigenspaces of the operator 
 \[ 
 \D_{B_0}  :  \Gamma (S) \to  \Gamma (S). 
\]
 \item[(ii)] the map $\widehat{r} : L^2_{k,  \delta}(S^+_{W[-\infty, 0] })  \to  L^2_{k-\frac{1}{2}}(S) $ is the restriction, 
 \item [(iii)]the operator  
 \[
 \widehat{p}^0_{-\infty} : L^2_{k-\frac{1}{2}}(S) \to \widehat{V}^0_{-\infty}  (Y, \C)
 \]
 is the $L^2$-projection to $\widehat{V}^0_{-\infty}  (Y)$. 
 \end{itemize}
 
\end{itemize}
We first prove the following proposition:

\begin{prop}\label{fredd}The following facts hold: 
\begin{itemize}
\item[(i)]
For any $ \delta \in \R$, the operator \eqref{Dirac} is Fredholm. 

\item[(ii)] Let $\delta_0$ be a positive real number given in \cref{rem deltazero}.  For any $\delta \in (0, \delta_0)$, the operator \eqref{AHS} is Fredholm. 
\end{itemize}
\end{prop}
\begin{proof}
Both statements follow from the standard patching argument of parametrixes of these operators. 
\begin{itemize}
\item
First, we prove (i). 
By \cref{allops}, since positive scalar curvature metrics are admissible, we see that the Dirac operator 
\[
D^+_{A_0}: L^2_{k,\delta} (\wt{X};S^+) \to L^2_{k-1,\delta} (\wt{X};S^-)
\]
is an isomorphism for any $\delta\in \R$, and we get a continuous inverse $P_\delta : L^2_{k-1,\delta} (\wt{X};S^-) \to L^2_{k,\delta} (\wt{X};S^+)$. 
By patching a local parametrix of \eqref{Dirac} near the boundary $Y$ and $P_\delta$, we obtain a parametrix of \eqref{Dirac}. This implies the conclusion. 
\item Next, we prove (ii).
By \cref{allops}, 
\[
0 \to L^2_{k+1,\delta} (i\Lambda^0(\wt{X})) \xrightarrow{d}  L^2_{k,\delta} (i\Lambda^1(\wt{X})) \xrightarrow{d^+} L^2_{k-1,\delta} (i\Lambda^+(\wt{X})) \to 0
\]
is an acyclic complex for $\delta \in \R \setminus \mathcal{D}$, where $ \mathcal{D}$ is a discrete subset of $\R $ given in \cref{fred}.
This implies that 
\[
d^+ + d^{*_{\delta}} : L^2_{k,\delta} (i\Lambda^1(\wt{X})) \to L^2_{k-1,\delta} (i\Lambda^+(\wt{X})) \oplus L^2_{k-1,\delta} (i\Lambda^0(\wt{X})) 
\]
is an isomorphism for $\delta \in \R \setminus \mathcal{D}$. Since $\mathcal{D}$ does not have accumulation points, there exists a small positive real number $\delta_0$ such that 
\[
(0, \delta_0) \cap \mathcal{D} = \emptyset.
\]
 Then the remaining part is the same as the proof of (i). 
\end{itemize}
\end{proof}

Set
 \[
W( Y ) := i\R^{b_0(Y)}\oplus d L^2_{k-1/2}(i\Lambda^0_{Y})
 \]
and consider the operators
 \begin{align*}
     &L_{W[-\infty, 0]}\oplus (p^0_{-\infty}\circ r) : \mathcal{U}_{k, \delta} \to L^2_{k-1, \delta}(i\Lambda^+\oplus S^-)\oplus V^0_{-\infty},\\
     &\widehat{L}_{W[-\infty, 0]}\oplus ( \widehat{p}^0_{-\infty}\circ \hat{r}) : L^2_{k, \delta}(i\Lambda^1\oplus S^+) \to L^2_{k-1, \delta}(i\Lambda^0\oplus i\Lambda^+\oplus S^-)\oplus \widehat{V}^0_{-\infty}
 \end{align*}
 over $W[-\infty, 0]$.
 Here $L_{W[-\infty, 0]}$ is defined in \eqref{L}, and $\widehat{L}_{W[-\infty, 0]}$ is defined by 
\[
\widehat{L}_{W[-\infty, 0]}(a,\phi) := (d^{\ast_{\delta}}a, d^+a, D^+_{A_0} \phi). 
 \]
 It follows from \cref{fredd} that the operator $\widehat{L}_{W[-\infty, 0]}\oplus ( \widehat{p}^0_{-\infty}\circ \hat{r})$ is Fredholm for all $\delta \in (0, \delta_0)$.
 
 \begin{prop}
 \label{prop:compari op}
 Let $\delta_0$ be the positive real number given in \cref{rem deltazero}. 
For any $\delta \in (0, \delta_0)$, we obtain
  \[
  \begin{cases} 
   \Ker (L_{W[-\infty, 0]}\oplus (p^0_{-\infty}\circ r))\cong \Ker (\widehat{L}_{W[-\infty, 0]}\oplus ( \widehat{p}^0_{-\infty}\circ \hat{r})), \\ 
      \operatorname{Coker} (L_{W[-\infty, 0]}\oplus (p^0_{-\infty}\circ r)) \cong \operatorname{Coker}  (\widehat{L}_{W[-\infty, 0]}\oplus ( \widehat{p}^0_{-\infty}\circ \hat{r})),
      \end{cases} 
   \]
   where $\Coker$ denotes the algebraic cokernel.
   In particular, $L_{W[-\infty, 0]}\oplus (p^0_{-\infty}\circ r)$ is Fredholm and the index of $L_{W[-\infty, 0]}\oplus (p^0_{-\infty}\circ r)$ coincides with that of $\widehat{L}_{W[-\infty, 0]}\oplus ( \widehat{p}^0_{-\infty}\circ \hat{r})$.
 \end{prop}
 
\begin{proof}
The proof is essentially the same as the proof in \cite{Khan15}. 
First, by the choice of $\delta$, \cref{fredd} implies that $d^{*_\delta} + d^+  +  \widehat{p}^0_{-\infty} \circ  \widehat{r} $ is Fredholm. 
Set
\[
\widehat{V}(Y) = \widehat{V} = i\Omega^0(Y)\oplus i\Omega^1(Y),
\]
and let
\[
\varpi: \widehat{V}(Y)\to W(Y)
\]
be the $L^2$-orthogonal projection, and consider an operator
\begin{align} \label{AHS''}\begin{split}
\widehat{L}_{W[-\infty, 0]}\oplus ((p^0_{-\infty} \oplus \varpi)\circ \hat{r}): L^2_{k, \delta}(i\Lambda^1_{W[-\infty, 0]}\oplus S^+)\\
\to L^2_{k-1, \delta}(i\Lambda^0_{W[-\infty, 0]} \oplus i\Lambda^+_{W[-\infty, 0]} \oplus S^-)\oplus V^0_{-\infty}(Y; \R)\oplus W(Y)
\end{split}
\end{align}
as an intermediary between the two operators in the statement of the \lcnamecref{prop:compari op}.

We first show that 
\begin{align}
\label{KerCoker1}
\begin{cases} 
\Ker (\widehat{L}_{W[-\infty, 0]}\oplus ((p^0_{-\infty} \oplus \varpi)\circ \hat{r})) \cong \Ker (\widehat{L}\oplus (\widehat{p}^0_{-\infty}\circ \widehat{r})) \\ 
 \operatorname{Coker} (\widehat{L}_{W[-\infty, 0]}\oplus ((p^0_{-\infty} \oplus \varpi)\circ \hat{r})) \cong  \operatorname{Coker} (\widehat{L}\oplus (\widehat{p}^0_{-\infty}\circ \widehat{r})). 
\end{cases}
\end{align}
Set
\[
V^\perp =V^\perp(Y) =i\Omega^0(Y)\oplus id\Omega^0(Y),
\]
and let
\[
l^\perp: V^\perp\to V^\perp
\]
be the operator defined by
\[
l^\perp=\begin{bmatrix}
0& -d^{*}\\
-d& 0
\end{bmatrix}.
\]
We denote the $L^2_{k-1/2}$-completion of $l^\perp$ by the same notation.
Then we have
\[
\widehat{V}=V\oplus V^\perp
\]
and
\[
\widehat{l}=l\oplus l^\perp.
\]
Let
$(V^\perp)^0 _{-\infty}$ be the span of non-positive eigenvectors of $l^\perp$.
As shown in \cite{Khan15}, the projection $\varpi: (V^\perp)^0 _{-\infty} \to W( Y )$ is an isomorphism, and
hence so is
\[
id_{V^0_{-\infty}} \oplus \varpi : \widehat{V}^0_{-\infty} = V^0_{-\infty} \oplus  (V^\perp)^0_{-\infty} \to V^0_{-\infty} \oplus W(Y).
\]
Thus we obtain the following commutative diagram between functional spaces over $W[-\infty, 0]$:
\[
  \begin{CD}
     L^2_{k, \delta}(i\Lambda^1_{}\oplus S^+)@>{\widehat{L}_{W[-\infty, 0]}\oplus \widehat{p}^{0}_{-\infty} \circ \widehat{r}}>> L^2_{k-1, \delta}(i\Lambda^0\oplus i\Lambda^+\oplus S^-)\oplus \widehat{V}^0_{-\infty} \\
@\vert  @V{id\oplus \varpi}V{\cong}V    \\
   L^2_{k, \delta}(i\Lambda^1\oplus S^+)  @>{\widehat{L}_{W[-\infty, 0]}\oplus ((p^0_{-\infty} \oplus \varpi)\circ \hat{r})}>>  L^2_{k-1, \delta}(i\Lambda^0\oplus i\Lambda^+\oplus S^-)\oplus V^0_{-\infty}\oplus  W( Y ).
  \end{CD}
\]
From this diagram we obtain the isomorphisms \eqref{KerCoker1}.
Moreover, as noted, it follows from \cref{fredd} that the operator $\widehat{L}_{W[-\infty, 0]}\oplus ( \widehat{p}^0_{-\infty}\circ \hat{r})$ is Fredholm.
Therefore this diagram implies that $ \widehat{L}_{W[-\infty, 0]}\oplus ((p^0_{-\infty} \oplus \varpi)\circ \hat{r})$ is also Fredholm.

The remaining task is to show that
\begin{align}
\label{KerCoker2}
\begin{cases} 
\Ker (L_{W[-\infty, 0]}\oplus (p^0_{-\infty}\circ r)) \cong \Ker (\widehat{L}_{W[-\infty, 0]}\oplus ((p^0_{-\infty} \oplus \varpi)\circ \hat{r})) \\  
\operatorname{Coker} (L_{W[-\infty, 0]}\oplus (p^0_{-\infty}\circ r)) \cong \operatorname{Coker}   (\widehat{L}_{W[-\infty, 0]}\oplus ((p^0_{-\infty} \oplus \varpi)\circ \hat{r})). 
\end{cases} 
\end{align}
The assertion of the \lcnamecref{prop:compari op} immediately follows from this and \eqref{KerCoker1}.
But applying the snake lemma to the following commutative diagram between functional spaces over $W[-\infty, 0]$, we can obtain \eqref{KerCoker2}:
\[
\begin{CD}
0  @. 0 \\ 
  @V VV  @VVV    \\
L^2_{k, \delta}(i\Lambda^1\oplus S^+)_{CC}@>{L_{W[-\infty, 0]}\oplus p^0_{-\infty}\circ r}>>  L^2_{k-1, \delta}(i\Lambda^+\oplus S^-)\oplus V^0_{-\infty} \\ 
 @V{} VV  @V{}VV    \\
L^2_{k, \delta}(i\Lambda^1\oplus S^+) @>{\widehat{L}_{W[-\infty, 0]}\oplus ((p^0_{-\infty} \oplus \varpi)\circ \hat{r}) }>>    L^2_{k-1, \delta}(i\Lambda^0\oplus i\Lambda^+\oplus S^-)\oplus V^0_{-\infty}\oplus W( Y ) \\ 
   @V{d^{*_\delta}\oplus \varpi \circ \widehat{r}} VV  @VVV    \\
 L^2_{k-1, \delta}(i\Lambda^0) \oplus W( Y )@=   L^2_{k-1, \delta}(i\Lambda^0)\oplus W( Y )   \\ 
 @V VV  @VVV    \\
 0  @. 0. \\ 
 \end{CD}
\]
\end{proof}

We consider a Riemannian manifold 
\[
\hat{W}[-\infty, 0]  :=  W[-\infty, 0] \cup (\R^{\geq 0} \times Y)
\]
obtained by gluing the half-cylinder $( \R^{\geq 0} \times Y, dt^2+ g_Y)$ with $W[-\infty, 0]$ along their boundary. 
We will compare formal adjoints $d^*$ for several weights, and would like to introduce a family of weight functions 
\[
\tau_{\delta, \delta'}: \hat{W}[-\infty, 0]\to \R_{\geq 0}
\]
 such that 
\[
(\tau_{\delta, \delta'})|_{{W}[-\infty, -1]} = \delta \tau \text{\quad  and\quad } (\tau_{\delta, \delta'})|_{[1,\infty) \times Y } = \delta'  t. 
\]
\begin{defi}
Let $(\delta, \delta') \in  \R^2$.
For a bundle $E$ which is periodic on ${W}[-\infty, 0]$ and cylindrical on $[0,\infty) \times Y$, we define the norm $\| - \|_{ L^2_{k, (\delta, \delta' ) } (E)} $ by
\[
\| f\|_{ L^2_{k, (\delta, \delta') } (E)}  :=  \|e^{\tau_{\delta, \delta'} } f \|_{L^2_{k } (E)} 
 \]
 and define $L^2_{k, (\delta, \delta' ) } (E)$ to be the completion of compactly supported sections with respect to $\| - \|_{ L^2_{k, (\delta, \delta' ) } (E)} $. 
\end{defi}

Note that the formal adjoint with respect to $L^2_{(\delta, \delta')}$ of $d$ is given as 
\[
d^{*_{(\delta, \delta')}} (w ) =e^{-\tau_{\delta, \delta'} }  d^* (e^{\tau_{\delta, \delta'} }  w). 
\]

We also consider the `sliced' Atiyah--Hitchin--Singer operator with APS-boundary condition: 
\begin{align}\label{AHS'}
\begin{split}
d^+  +  {p}^0_{-\infty} \circ  {r} :  L^2_{k, \delta}(i \Lambda_{W[-\infty, 0] }^1)_{CC}   
 \to  L^2_{k-1,  \delta}(i \Lambda_{W[-\infty, 0] }^+  )\oplus {V}^0_{-\infty}  (Y; \R). 
 \end{split}
 \end{align}
We calculate the kernel and the cokernel of \eqref{AHS'} :
\begin{theo} \label{cohomology} There exists $\delta_1 >0$ such that for any $\delta \in (0, \delta_1)$,
the operator \eqref{AHS'} is an isomorphism. 
\end{theo}
We take the constant $\delta_1$ to be smaller than $\delta_0$ given in \cref{rem deltazero}. 
The rest of this subsection is devoted to prove \cref{cohomology}.

To prove \cref{cohomology}, it is sufficient to prove the operator \eqref{AHS} is invertible for a sufficiently small $\delta>0$. 
First we shall calculate the kernel and the cokernel of 
\begin{align}
    \label{eq: AHS''}
    \begin{split}
d^{*_\delta} + d^+  +  \widehat{p}^0_{-\infty} \circ  \widehat{r} :  L^2_{k, \delta}(i \Lambda_{W[-\infty, 0] }^1)  \\ 
 \to  L^2_{k-1,  \delta}(i\Lambda_{W[-\infty, 0] }^0   \oplus \Lambda_{W[-\infty, 0] }^+  )\oplus \widehat{V}^0_{-\infty}  (Y; \R). 
 \end{split}
 \end{align}
 The following lemma can be proved by considering the similar discussion given in \cite{APSI}. 
 \begin{lem}\label{cohomology2}We have the following identifications: 
 \[
 \begin{cases} 
 \Ker(d^{*_\delta} + d^+  +  \widehat{p}^0_{-\infty} \circ  \widehat{r})  = \Set{ a \in L^2_{k, (\delta, 0) }(i \Lambda_{\hat{W}[-\infty, 0] }^1)  | d^{*_{(\delta,0)}} a =0,   d^+a =0 },  \\
 \operatorname{Coker}(d^{*_\delta} + d^+  +  \widehat{p}^0_{-\infty} \circ  \widehat{r})   \\ 
  = \Set{ (0,  b ) \in  L^2_{k-1,(\delta, 0)}(i\Lambda_{\hat{W}[-\infty, 0] }^0   \oplus \Lambda_{\hat{W}[-\infty, 0] }^+  ) |   d^{*_{(\delta,0)}} b=0   } . 
 \end{cases} 
 \]

 \end{lem} 
 \begin{proof}
 By the same discussion in  \cite[Proposition 3.11]{APSI}, 
 a solution under the spectral boundary condition can be identified with an $L^2$-solution on a cylindrical end manifold.
 Thus one has an isomorphism
 \[
  \Ker (d^{*_\delta} + d^+  +  \widehat{p}^0_{-\infty} \circ  \widehat{r})   \cong \Set{ a \in L^2_{k, (\delta, 0) }(i \Lambda_{\hat{W}[-\infty, 0] }^1)  | d^{*_{(\delta,0)}} a =0,   d^+a =0 },
  \]
  and the cokernel can be written by using extended $L^2_k $-solutions: 
 \begin{align*} 
 & \operatorname{Coker}  (d^{*_\delta} + d^+  +  \widehat{p}^0_{-\infty} \circ  \widehat{r} )   \\ 
 \cong& 
 \Set{ 
 (a, b ) 
   |  
  \begin{matrix}
  & d^{*_{(\delta,0)}} b=0 , \ d a=0, \\
  & (a,b)|_{W[-\infty, 0]} \in L^2_{k-1,  \delta }(i\Lambda_{W[-\infty, 0] }^0   \oplus \Lambda_{W[-\infty, 0] }^+  ) , \\
  &  (a- c , b ) \in L^2_{k-1  }(i\Lambda_{[0, \infty) \times Y  }^0   \oplus \Lambda_{[0, \infty) \times Y }^+  ), \ \exists c \in \R    
  \end{matrix}
  } \\
  \subset & L^2_{k-1,  \text{loc}}(i\Lambda_{\hat{W}[-\infty, 0] }^0   \oplus \Lambda_{\hat{W}[-\infty, 0] }^+  ) . 
 \end{align*} 
 Here we used $H^*(Y; \R) \cong H^* (S^3; \R)$. 
On the other hand, $da=0$ implies $a$ is a constant and $c$ should be zero. This gives the conclusion. 
 \end{proof}

\begin{lem}
\label{identify}
For $\delta$ sufficiently small, the space 
\[
\Set{ a \in L^2_{k, (\delta, 0) }(i \Lambda_{\hat{W}[-\infty, 0] }^1)   |  d^{*_{(\delta,0)}} a =0,   da =0 }
\]
can be identified with the middle cohomology of 
\[
L^2_{k+1, (\delta, \delta) }(i \Lambda_{\hat{W}[-\infty, 0] }^0)  \xrightarrow{d} L^2_{k, (\delta, \delta) }(i \Lambda_{\hat{W}[-\infty, 0] }^1)  \xrightarrow{d}  L^2_{k-1, (\delta, \delta) }(i \Lambda_{\hat{W}[-\infty, 0] }^2).
\] 
\end{lem} 
\begin{proof}
By the exponential decay result, one can see the correspondence 
\[
f \mapsto e^{\tau_{\delta, \delta}  -\tau_{\delta, 0 } } f
\]
 gives an identification
\[
\Set{ a \in L^2_{k, (\delta, 0) }(i \Lambda_{\hat{W}[-\infty, 0] }^1)   |  e^{-\tau_{\delta, 0 } }  d^* (e^{\tau_{\delta, 0 } }  b) =0,   da =0 } 
\]
\[
\to \Set{ a \in L^2_{k, (\delta, \delta ) }(i \Lambda_{\hat{W}[-\infty, 0] }^1)   |  e^{-\tau_{\delta, \delta} }  d^* (e^{\tau_{\delta, \delta} }  b) =0 ,   da =0 }. 
\]

On the other hand,
for an appropriate $\delta$, the complex 
\begin{align}\label{complex}
L^2_{k, (\delta, \delta) }(i \Lambda_{\hat{W}[-\infty, 0] }^0)   \xrightarrow{d}  L^2_{k, (\delta, \delta) }(i \Lambda_{\hat{W}[-\infty, 0] }^1) \xrightarrow{d}  L^2_{k, (\delta, \delta) }(i \Lambda_{\hat{W}[-\infty, 0] }^2)
\end{align}
is Fredholm, and we can identify the kernel of 
\[
L^2_{k, (\delta, \delta) }(i \Lambda_{\hat{W}[-\infty, 0] }^1)  \xrightarrow{d+d^{*_{(\delta, \delta)} } } L^2_{k-1, (\delta, \delta) }(i \Lambda_{\hat{W}[-\infty, 0] }^0\oplus i \Lambda_{\hat{W}[-\infty, 0] }^2)
\]
with the middle cohomology of \eqref{complex}. 
\end{proof} 

The exponential decay result enables us to prove the following correspondence: 
\begin{lem}
\label{lem:deldel del zero}
For $\delta>0$ sufficiently small, 
the space 
\[
\Set{ b \in  L^2_{k-1, (\delta, 0)  }(i \Lambda_{\hat{W}[-\infty, 0] }^+  )  |    d^{*_{(\delta, 0)} } b=0   } 
\]
 can be identified with 
\[
\Set{ b \in  L^2_{k-1, (\delta, \delta)  }(i \Lambda_{\hat{W}[-\infty, 0] }^+  ) |    d^{*_{(\delta, \delta)}}  b=0   } .
\]
\end{lem}
\begin{proof}
The proof is the same as in that of \cref{identify}. 
\end{proof}

\begin{lem}
\label{lem: ker coker key}
We have that
\begin{align}
\label{eq: ker dast d}
\Set{ a \in L^2_{k, (\delta, 0) }(i \Lambda_{\hat{W}[-\infty, 0] }^1)   |  d^{*_{(\delta,0)}} a =0,   da =0 } = \{0\},
\end{align}
and one can construct an injective homomorphism 
\begin{align*}
&\Set{ (0,  b ) \in  L^2_{k-1, (\delta, 0)  }(i\Lambda_{\hat{W}[-\infty, 0] }^0   \oplus \Lambda_{\hat{W}[-\infty, 0] }^+  )  |    d^{*_{(\delta,0)}} b=0   } \\ 
&\to  H^2 (W[-n, 0]  \cup [0,n]\times Y , \partial (W[-n, 0]  \cup [0,n]\times Y)  ; \R ) 
\end{align*}  
for $n$ sufficiently large. 
\end{lem} 
\begin{proof}
This is essentially the same argument given in \cite[Proof of Proposition 5.1]{T87}. 
To prove the first assertion, we construct an injective homomorphism
\begin{align*}
&\{ a \in L^2_{k, (\delta, 0) }(i \Lambda_{\hat{W}[-\infty, 0] }^1) \  | \ d^{*_{(\delta,0)}} a =0,   da =0 \}\\
 &\to H^1(W[-2, 0] \cup [0,2]\times Y ,  \partial  (W[-2, 0] \cup [0,2]\times Y ) ; \R),
\end{align*}
and show that this map factors through $\{0\}$.
The proof uses the condition $H^1(W_n; \R)=0$.
This map is defined by choosing a bump function $\beta : \hat{W}[-\infty, 0]  \to \R$ such that
\begin{itemize}
\item $\beta |_{W[-1, 0] \cup [0,1]\times Y } =0$ and 
\item  $\beta |_{ \hat{W}[-\infty, 0] \setminus \operatorname{int} W[-2, 0] \cup [0,2]\times Y } =1$. 
\end{itemize}
For a given $w \in \Ker d \subset L^2_{k, (\delta, \delta) }(i \Lambda_{\hat{W}[-\infty, 0] }^1) $, one can assume $w$ is smooth. 
Since $H^1(W[-\infty, -1]; \R)=0$ and $H^1(Y; \R)=0$, one can choose smooth functions $f_-$ on $W[-\infty, -1]$ and $f_+$ on $[1, \infty) \times Y$ such that 
\[
w|_{W[-\infty, -1]}  = d f_- \text{ and } w|_{[1, \infty) \times Y} = d f_+ . 
\]
Since $w|_{W[-\infty, -1]}  = d f_-  \in L^2_{k, \delta}$ and $w|_{[1, \infty) \times Y} = d f_+ \in L^2_{k, 0}$ , up to adding constants, one can assume 
\[
f_- \in L^2_{k+1, \delta}(W[-\infty, -1]) \text{ and } f_+  \in L^2_{k+1, \delta}([1, \infty) \times Y).  
\]
Define 
\[
\phi ([w]) :=   w- d (\beta f_+ + \beta f_-). 
\]
One can see that $\phi$ induces a homomorphism 
\[
\phi : ( \Ker d \subset L^2_{k, (\delta, \delta) }(i \Lambda_{\hat{W}[-\infty, 0] }^1))  / d ( L^2_{k, (\delta, \delta) }(i \Lambda_{\hat{W}[-\infty, 0] }^0) ) \to
\]
\[
 H^1(W[-2, 0] \cup [0,2]\times Y ,  \partial W[-2, 0] \cup [0,2]\times Y ; \R).  
\]
The same argument given in \cite[(5.8)-(5.10)]{T87} shows that $\phi$ is an injection. 
Under our assumption, the cohomology
\[
H^1(W[-2, 0] \cup [0,2]\times Y ,  \partial W[-2, 0] \cup [0,2]\times Y ; \R)
\] is generated by 
\[
 [d  g_0]  \in H^1(W[-2, 0] \cup [0,2]\times Y ,  \partial W[-2, 0] \cup [0,2]\times Y ; \R)
 \]
 which is constant near $ \partial (W[-2, 0] \cup [0,2]\times Y) $ and $g_0|_{Y^-_{-3}} \neq g_0 |_{\{2\} \times Y}$, where $\partial W_i = Y^-_i \cup Y^+_i$ and $-Y=Y^-_i$ as oriented manifolds.
 
Next we show that $\im\phi=\{0\}$.
 Suppose that $[dg_0] = \phi ([w])$. Then we have 
 \[
 \phi ([w]) =   w- d f= dg_0  +d g' 
 \]
 for some $g' \in \Om^0 (W[-2, 0] \cup [0,2]\times Y ,  \partial W[-2, 0] \cup [0,2]\times Y)$ and for some $f \in L^2_{k+1, (\delta, \delta) }(\hat{W}[-\infty, 0])$. Up to image $d$, we can assume $d^{*_{(\delta, \delta) } } w=0$. Thus we have 
 \[
- d^{*_{(\delta, \delta) } } df = d^{*_{(\delta, \delta) } } dg_0  + d^{*_{(\delta, \delta) } }d g' ,
 \]
 here we consider $g_0$ and $g'$ as constant extensions on the ends.
This implies 
\[
-f= g_0+  g'+ \text{constant}.
\]
Since $f$ goes to $0$ on the ends, this gives a contradiction. 

Combining this with the injectivity of $\phi$, 
we have that the domain of $\phi$ is $\{0\}$.
Here \cref{identify} implies that the domain of $\phi$ is isomorphic to the left-hand side of \eqref{eq: ker dast d}, and now the first assertion of the \lcnamecref{lem: ker coker key} follows.

On the second assertion, a homomorphism 
\begin{align*}
\phi_n : &\{ (0,  b ) \in  L^2_{k-1, (\delta, 0)  }(i\Lambda_{\hat{W}[-\infty, 0] }^0   \oplus \Lambda_{\hat{W}[-\infty, 0] }^+  )\  |\    d^{*_{(\delta,0)}} b=0   \} \\ 
&\to  H^2 (W[-n, 0]  \cup [0,n]\times Y , \partial (W[-n, 0]  \cup [0,n]\times Y)  ; \R ) 
\end{align*}  
is given as follows.
First note that the domain of $\phi_n$ is identified with the corresponding functional space for the weight $(\delta, \delta)$ because of \cref{lem:deldel del zero}.
Take a sequence of bump functions $\beta_n : \hat{W}[-\infty, 0]  \to \R$ satisfying 
\[
\beta_n|_{W[-n+1, 0] \cup [0, n-1] \times Y} = 1\ ,  \beta_n|_{W[ -\infty , -n] \cup [n, \infty ] \times Y}=0, \text{ and }|d\beta_n|_{C^0}<C.
\]
For a given $w \in \{ b \in  L^2_{k-1, (\delta, \delta)  }(i \Lambda_{\hat{W}[-\infty, 0] }^+  )\  |\    d^{*_{(\delta,\delta)}} b=0   \} $,  one can see
\[
d e^{\hat{\tau}\delta } w=0. 
\]
Since $H^2(W[-n-2, -n+2]; \R) =0$ and $H^2([n-2, n+2] \times Y; \R)=0$, one can choose 
$\gamma_-$ and $\gamma_+$ such that 
\[
e^{\hat{\tau}\delta } w|_{W[-n-2, -n+2]} = d \gamma^n_- \text{ and } e^{\hat{\tau}\delta } w|_{[n-2, n+2] \times Y} = d \gamma^n_+.  
\]
Define 
\[
\phi_n (w) := \begin{cases} 
e^{\hat{\tau}\delta } w \text{ on } W[-n-2, 0] \cup [0, n+1] \times Y \\
d (\beta_n\gamma^n_- +  \beta_n\gamma^n_+ ) \text{ on } W[-n-2, -n+2] \cup [n-2, n+2] \times Y \\
0 \text{ otherwise } . 
\end{cases}
\]
The proof of the injectivity of $\phi_n$ is the same as the proof of \cite[Lemma 5.4]{T87}.
\end{proof}

\begin{proof}[Proof of \cref{cohomology}]
Note that \cref{prop:compari op} gives isomorphisms of the kernels and cokernels between the operators \eqref{AHS'} and \eqref{eq: AHS''}, since the operator dealt with in \cref{prop:compari op} is the direct sum of a real operator and a complex operator.
Using this and \cref{cohomology2}, to show the \lcnamecref{cohomology}, it suffices to see that
\begin{align}
\label{eq: vanish1}
\Set{ a \in L^2_{k, (\delta, 0) }(i \Lambda_{\hat{W}[-\infty, 0] }^1)   |  d^{*_{(\delta,0)}} a =0,   d^+a =0  } = \{0\} 
\end{align}
and 
\begin{align}
\label{eq: vanish2}
\Set{ (0,  b ) \in  L^2_{k-1, (\delta, 0)  }(i\Lambda_{\hat{W}[-\infty, 0] }^0   \oplus \Lambda_{\hat{W}[-\infty, 0] }^+  )|    d^{*_{(\delta,0)}} b=0   } =\{0\}  . 
\end{align}
By integration by parts, one has 
\begin{align*}
&\Set{ a \in L^2_{k, (\delta, 0) }(i \Lambda_{\hat{W}[-\infty, 0] }^1)   |  d^{*_{(\delta,0)}} a =0,   d^+a =0 }\\
=& \Set{ a \in L^2_{k, (\delta, 0) }(i \Lambda_{\hat{W}[-\infty, 0] }^1)   |  d^{*_{(\delta,0)}} a =0,   da =0 }. 
\end{align*}
The vanishing \eqref{eq: vanish1} follows from this and the first assertion of \cref{lem: ker coker key}.

Our assumption implies $H^2 (W[-n, 0], \partial W[-n, 0] )=0$, and the vanishing \eqref{eq: vanish2} follows from the second assertion of \cref{lem: ker coker key}.
\end{proof}

In the proof of \cref{real main}, we also consider the `right-periodic' manifold $W[0,\infty]$.
Fix a Riemann metric $g_{W[0,\infty] }$ on $W[0,\infty]$ such that 
\begin{itemize} 
\item $g_{W[0,\infty] } |_{W[1,\infty]}$ is periodic and PSC, and 
\item $ g_{W[0,\infty] }$ is product metric near $\partial  W[0,\infty] =- Y$. 
\end{itemize}
Let us consider the following operators: 
\begin{itemize}
    \item the Atiyah--Hitchin--Singer operator with APS-boundary condition: 
\begin{align}\label{-AHS'}
\begin{split}
d^+  +  {p}^0_{-\infty} \circ  {r} :  L^2_{k, \delta}(i \Lambda_{W[0,\infty]  }^1)_{CC}   
 \to  L^2_{k-1,  \delta}(i \Lambda_{W[0,\infty]  }^+  )\oplus {V}^0_{-\infty}  (-Y; \R), 
 \end{split}
 \end{align}
    \item the linearlization of the Seiberg--Witten equation
\begin{align}\label{-}
    L_{W[0,\infty] }\oplus (p^0_{-\infty}\circ r) : \mathcal{U}_{k, \delta}  \to L^2_{k-1, \delta}(i\Lambda^+\oplus S^-)\oplus V^0_{-\infty}(-Y). 
\end{align}
\end{itemize}

\begin{theo}\label{cohomology1}
There exists $\delta'_1 >0$ such that for any $\delta \in (0, \delta'_1)$, the followings are true: 
\begin{itemize}
    \item[(i)] the operator \eqref{-} is Fredholm, and 
    \item[(ii)] the operator \eqref{-AHS'} is isomorphism. 
\end{itemize}
\end{theo}
\begin{proof}
The proof of (i) is the same as that of \cref{fredd}. 
The proof of (ii) is also essentially the same as that of \cref{cohomology}. 
\end{proof}

\subsection{Global slice theorem}
In this subsection we prove the global slice theorem in our situation. We follow the method given in \cite{IT20}. In \cite{IT20}, for 4-manifolds with conical end, a global slice theorem is given and the essentially same method can be applied to our situation. 

The following proposition is a key lemma to prove the global slice theorem: 
\begin{prop}\label{Slice}
There exists a small positive number $ \delta_2$ such that for any positive real number $0< \delta \leq  \delta_2$, 
\begin{align}
L^2_{k,  \delta} (i \Lambda^1_{W[-\infty, 0]} ) = L^2_{k,  \delta} (i\Lambda^1_{W[-\infty, 0]} )_{CC} \oplus d L^2_{k+1,  \delta} (i\Lambda^0_{W[-\infty, 0]} ) . 
\end{align}
\end{prop}
This proposition corresponds to \cite[Proposition 3.5]{IT20}. 
\begin{proof}
The proof is essentially same as the proof of  \cite[Proposition 3.5]{IT20}.
We first prove
\begin{align} \label{intersection}
L^2_{k, \delta} (i\Lambda^1_{W[-\infty, 0] } )_{CC}\cap  d L^2_{k+1, \delta} (i\Lambda^0_{W[-\infty, 0] } )  = \{0\}. 
\end{align}
 However, the proof of \eqref{intersection} is the same as the proof of (21) in \cite[Proposition 3.5]{IT20}, and we omit this. 
 
Next, we will see 
\[
L^2_{k, \delta} ( i\Lambda^1_{W[-\infty, 0] } ) = L^2_{k, \delta} (i\Lambda^1_{W[-\infty, 0] } )_{CC} +  d L^2_{k+1, \delta} (i\Lambda^0_{W[-\infty, 0] } ) . 
\]
 We need to prove that, for any $\al \in L^2_{k, \delta} (i \Lambda^1_{W[-\infty, 0] } )$, there exists $\xi \in L^2_{k+1, \delta} (i\Lambda^0_{W[-\infty, 0] } )$ such that $\al- d \xi  \in L^2_{k, \delta} (i\Lambda^1_{W[-\infty, 0] } )_{CC}$, i.e. 
 \begin{align*} 
 d^{*_\delta} d \xi  =  d^{*_\delta}\alpha \\
 d^{*} {\bf t} d \xi  =  d^{*} {\bf t} \alpha 
 \end{align*}
 hold. These equations are equivalent to 
  \begin{align*} 
 \Delta_\delta \xi  =  d^{*_\delta}\alpha \\
 {\bf t}  \xi  =  G_{Y } d^{*} {\bf t} \alpha , 
 \end{align*}
where $G_{Y }$ is the Green operator on $Y $. 
Therefore we need to prove surjectivity of the map 
\[
\Delta_\delta(W[-\infty, 0] , \partial ) : L^2_{k+1, \delta} (i\Lambda^0_{W[-\infty, 0] } )  \to L^2_{k-1, \delta} (i\Lambda^0_{W[-\infty, 0] } ) \oplus  L^2_{k+ \frac{1}{2}} (i \Lambda^0_{Y } ),  
\]
defined by 
\[
\Delta_\delta (W[-\infty, 0] , \partial ) \xi = ( \Delta_\delta \xi , {\bf t}\xi ). 
\]
In order to prove this, we use the excision principle and reduce the surjectivity of $ \Delta_\al(W[-\infty, 0] , \partial )$ to calculations of indexes for several Laplacian operators. 
The calculation of indicies of Laplacian operators are also given in  \cite[Proposition 3.5, page 18]{IT20}. We can confirm the surjectivity of $\Delta_\delta(W[-\infty, 0] , \partial ) $ and obtain the conclusion. 
\end{proof}

\cref{Slice} implies the following global slice theorem: 
\begin{lem}\label{global}Let $\delta_2$ be the constant given in \cref{Slice}. Then, for $\delta \in (0, \delta_2)$, there is a $\G_{k+1, \delta}(W[-\infty, 0] )$-equivariant diffeomorphism 
\[
\mathcal{U}_{k, \delta}(W[-\infty, 0] ) \cong L^2_{k,  \delta} (i\Lambda^1_{W[-\infty, 0]} )_{CC} \times \G_{k+1, \delta}(W[-\infty, 0] ) . 
\]
\end{lem}
The proof is the essentially same as in the case of closed 4-manifolds.

\subsection{Dirac index on $W[-\infty,0]$} \label{section: Dirac index}
\label{Dirac index on Winf}

In this \lcnamecref{Dirac index on Winf}, we shall calculate the spin Dirac index $\ind_{\C}D^+_{W[-\infty,0]}$ on the half-periodic $4$-manifold $W[-\infty,0]$:


\begin{prop}
\label{prop: excision}
Assuming that a PSC metric is equipped with $X$, we have
\begin{align}
\label{eq: excision aim}
\ind_{\C}D^+_{W[-\infty,0]} = \lambda_{SW}(X, \fraks) + n(Y,\frakt,g), 
\end{align}
where $\ind_{\C}D^+_{W[-\infty,0]}$ means the index of the Dirac operator under the APS-boundary condition and $n(Y,\frakt,g)$ is given in \eqref{n}. 
\end{prop}

Before proving \cref{prop: excision}, we note a few lemmas:

\begin{lem}
\label{prop: excision1}
Let $M_{1}, M_{2}$ be compact spin $4$-manifolds with common boundary $Y$ with orientation $\del M_{1} = Y = - \del M_{2}$.
Equip $M_{1}, M_{2}$ with metrics so that the metrics are the product metric
\[
dt^2 + \pr^* g_Y
\]
near the boundary for a Riemann metric $g_Y$ on $Y$, where $t$ is a collar coordinate of the product neighborhood and $\pr$ means the projection from the collar neighborhoods of $Y$ to $Y$. 
Then we have 
\[
\ind_{\C}D^+_{M_{1}} + \ind_{\C}D^+_ {M_{2}} + \dim_{\C} \Ker\D 
= \ind_{\C}D^+_ {M_{1}\cup_{Y} M_{2}},
\]
where $\D$ is the $3$-dimensional Dirac operator on $Y$ and $\ind_{\C}D^+_ {M_i}$ denotes the index of the Dirac operator under the APS-boundary condition. 
\end{lem}

\begin{proof}
This can be checked by the Atiyah--Singer--Patodi index theorem~\cite{APSI} immediately, but we give a bit more direct proof to make clear the following \cref{prop: excision2}.
 
We follow an argument given in Donaldson's book~\cite{Do02}, mainly \cite[Subsubsection~3.3.1]{Do02}.
For $\alpha \in \R$ which is not a spectrum of $\D$,
denote by $\ind_{\C}D^+_ {M_{1}, \alpha}$ the Fredholm index defined using the weighted Sobolev norm described as 
\[
\|f\|_{L^{2}_{k, \alpha}} = \|e^{\alpha t}f\|_{L^{2}_{k}}
\]
at the end of $M_{1} = M_{1} \cup [0,\infty) \times Y$.
Take $\alpha>0$ so that $|\alpha|$ is smaller than the absolute value of the smallest non-zero eigenvalue of $\D$. 
Then we obtain
\[
\ind_{\C}D^+_ {M_{1}\cup_{Y} M_{2}} = \ind_{\C}D^+_ {M_{1},\alpha} + \ind_{\C}D^+_ {M_{2},-\alpha}
\]
by the gluing formula, Equation~(3.2) of \cite{Do02}.
Hence it suffices to show that
\begin{align}
\label{eq: ind weigh ker}
\ind_{\C}D^+_ {M_{1},\alpha} + \ind_{\C}D^+_ {M_{2},-\alpha}
= \ind_{\C}D^+_ {M_{1}} + \ind_{\C}D^+_ {M_{2}} + \dim_{\C} \Ker\D.
\end{align}
By the definition of the APS-boundary condition,
we have 
\[
\ind_{\C}D^+_ {M_{1},\alpha} = \ind_{\C}D^+_ {M_{1}}.
\]
On the other hand, we have that 
\[
\ind_{\C}D^+_ {M_{2},-\alpha} = \ind_{\C}D^+_ {M_{2}} + \dim_{\C} \Ker\D
\]
by \cite[Proposition~3.10]{Do02}, which is shown considering a certain ordinary equation \cite[Lemma~3.11]{Do02} corresponding to the cylinder $(-\infty, \infty) \times Y$ appearing the neck stretching of $M_{1}\cup_{Y} M_{2}$.
Now we have checked \eqref{eq: ind weigh ker} and this completes the proof.
 \end{proof}
 
 The proof of \cref{prop: excision1} involves only near the neck of $M_{1}\cup_{Y} M_{2}$.
 Even if we replace $M_{1}$ with a manifold with an additional end, we obtain a similar result as far as we work in Fredholm setting.
This makes clear the following \lcnamecref{prop: excision2}:
  
 \begin{lem}
\label{prop: excision2}
Let $M$ be a compact spin manifold bounded by $Y$ with the orientation $\del M = -Y$.
Equip $M$ with a metric so that the metrics are product metrics near the boundary.
Then we have
\begin{align}
\label{eq: excision}
\ind_{\C}D^+_{W[-\infty,0]}  + \ind_{\C}D^+_ {M} + \dim_{\C} \Ker\D = \ind_{\C}D^+_ {M_{\infty}},
\end{align}
where 
\[
M_{\infty} =  \cdots  \cup_{Y} W\cup_{Y} W \cup_{Y} M.
\]
\end{lem}

Now we are ready to prove \cref{prop: excision}.

\begin{proof}[Proof of \cref{prop: excision}]
Take a compact spin bound $M$ of $Y$ with the orientation $\del M = -Y$.
Take a metric on $M$ so that the metrics are product metrics near the boundary.


Now we shall check
\begin{align}
\label{eq: ind infty}
\ind_{\C}D^+_ {M_{\infty}}
= -\frac{\sigma(M)}{8} +\lambda_{SW}(X,\fraks).
\end{align}
Note the sign:
this comes from the orientation of $M$ with $\del M = -Y$.
Indeed, by \cite[Lemma~2.21]{Lin19}, it follows from the existence of PSC metric on $X$ that
\begin{align}
\label{Jimila1}
-\lambda_{SW}(-X,\fraks) =  \lambda_{SW}(X,\fraks).
\end{align}
(Precisely, $X$ is supposed to be an integral homology $S^{1} \times S^{3}$ in \cite{Lin19}, but the proof of \cite[Lemma~2.21]{Lin19} is valid also for rational homology $S^{1} \times S^{3}$'s without any changes.)
On the other hand, for a PSC metric $g$ on $X$, we have
\begin{align}
\label{Jimila3}
  \lambda_{SW}(-X,\fraks) 
 = -w(-X,g,0) 
 = -\ind_{\C} D^+_{M_{\infty}} - \frac{\sigma(M)}{8}.
\end{align}
Equation~\eqref{eq: ind infty} is deduced from \eqref{Jimila1} and \eqref{Jimila3}.

On the other hand, we also have
\begin{align}
\label{eq: ind hat}
\ind_{\C}D^+_ {M}
= -\frac{\sigma(M)}{8} -n(Y,\frakt,g)-\dim_{\C}\Ker\D.
\end{align}
Indeed, it follows that
\begin{align}
\label{eq: ind hat1}
\ind_{\C}D^+_ {-M} + \ind_{\C}D^+_ {M} +\dim_{\C} \Ker\D = 0
\end{align}
 because of \cref{prop: excision1} and 
\[
\ind_{\C}D^+_ {-M\cup_{Y} M} = \frac{\sigma(-M)}{8} + \frac{\sigma(M)}{8}=0.
\]
By the definition of $n(Y,\frakt,g)$, we have 
\begin{align}
\label{eq: ind hat2}
n(Y,\frakt,g) = \ind_{\C}D^+_ {-M} + \frac{\sigma(-M)}{8}.
\end{align}
Equation~\eqref{eq: ind hat} is deduced from \eqref{eq: ind hat1}, \eqref{eq: ind hat2}.
%

Combining \cref{prop: excision2} with \eqref{eq: ind infty} and \eqref{eq: ind hat}, we obtain the desired equality \eqref{eq: excision aim}.
\end{proof}

\section{The boundedness result} \label{section: The boundedness result}
In this section, we prove a certain boundedness result in order to construct Bauer--Furuta type invariant. We mainly follow the methods given in \cite{Ma03, Kha15}. The situation is similar to that in \cite{IT20}, which gives a Bauer-Furuta invariant for 4-manifolds with conical end.

Our main result in this section is: 
\begin{theo}\label{bounded}
There exists $\delta_{3} >0$ and a constant $R>0$ such that the following conclusion holds. 
Let $\delta$ be an element in $(0,\delta_3]$. 
Suppose that a pair $(x,y)$ of 
 \[
x \in  \cU_{k, \delta}  ({W[-\infty, 0]}) 
\]
and $y: [0,\infty) \to V(Y) $
satisfy the following conditions: 
\begin{itemize}
\item[(i)] the element $x+ (A_0, \Phi_0)$ is a solution to the equation \eqref{SW eq} on $W[-\infty, 0]$, 
\item[(ii)]the element $y$ is a solution to the Seiberg--Witten equations on $\R^{\geq 0}\times Y$,
\item[(iii)] $y$ is temporal gauge, i.e. $d^{*}b(t)=0$ for each $t$, where $y(t)= (b(t), \psi(t))$, and $y$ is of finite type, 
\item[(iv)] $x|_{Y} = y (0)$, and 
\item[(v)]  $ | \lim_{t\to \infty }  CSD ( y(t) ) | < \infty$.
\end{itemize}
  Then we have the following universal bounds: 
  \[
  \| x\|_{L^2_{k, \delta} } < R  \text{ and } \|y(t)\|_{L^2_{k-\frac{1}{2}}} < R \  (\forall t  \geq 0). 
  \]
\end{theo}

In order to prove \cref{bounded}, we use several corresponding notions used in \cite{Ma03}. 

\begin{defi} We consider a Riemannian manifold 
\[
\hat{W}[-\infty, 0]  =  W[-\infty, 0] \cup (\R^{\geq 0} \times Y)
\]
obtained by gluing the half-cylinder $( \R^{\geq 0} \times Y, dt^2+ g_Y)$ and $W[-\infty, 0]$ along their boundary. 
A solution $(A, \Phi)$ to the Seiberg--Witten equations on $\hat{W}[-\infty, 0]$ is called {\it ${W}[-\infty, 0]$-trajectories}.
 If a $W[-\infty, 0]$-trajectory $(A, \Phi)$ satisfies 
\[
\sup_{t \in \R^{\geq 0}} |CSD (A|_{\{t\} \times Y } ) |< \infty \text{ and } \|\Phi\|_{C^0( \R^{\geq 0} \times Y )} < \infty, 
\]
then $(A, \Phi)$ is called a {\it finite type} $W[-\infty, 0]$-trajectory. 
\end{defi}

Let us note the following boundedness result:

\begin{theo}\label{exp} Let $C$ be a positive real number and
\[
(A, \Phi)  \in (A_0, 0 ) +  L^2_{k, \delta}( i\Lambda_{W[-\infty, 0] \cup [0, 1] \times Y }^1 ) \oplus L^2_{k, \delta} ( S^+_{W[-\infty, 0]\cup [0, 1] \times Y })
\]
be a solution to $\mathcal{F} (A, \Phi) = 0$ such that 
\[
\mathcal{E}^{top} (A, \Phi)  \leq C
\]
and 
\[
(A, \Phi) \in \mathcal{U}_{k, \delta} ( {W}[-\infty, 0]). 
\]

Then, there exists $\delta_3$ such that for any $\delta \in (0, \delta_3)$, the inequality 
\[
\| (A, \Phi) - (A_0, 0) \|_{L^2_{k, \delta} (W[-\infty, 0] )} \leq D(C)
\]
holds, where $D(C)$ is a constant depending only on $D$.

\end{theo}


\begin{proof}
We compare gauge transformations constructed by J.~Lin \cite[Subsection~4.2]{Lin19} with the global slice obtained in \cref{Slice}.
The proof of \cite[Lemma 4.10]{Lin19} implies that there exists a constant $\delta'_3$ and a gauge transformation $g'$ on $W[-\infty, 0]$  such that for any $\delta \in (0, \delta_3')$,
\[
\|(g')^* (A, \Phi) - (A_0, 0) \|_{L^2_{k, \delta} (W[-\infty, 0] )} \leq D(C). 
\]
Define 
\[
\delta_3 := \min \{\delta_3', \delta_2\} . 
\]
On the other hand, by \cref{global}, the map obtained by giving a slice 
\[
\mathcal{U}_{k, \delta}(W[-\infty, 0] ) \xrightarrow{\cong}  L^2_{k,  \delta} (i\Lambda^1_{W[-\infty, 0]} )_{CC} \times \G_{k+1, \delta}(W[-\infty, 0] )
\]
is continuous. 
This implies there is a gauge transformation $g$ such that
\[ 
g^* (A, \Phi) \in \mathcal{U}_{k, \delta} ( {W}[-\infty, 0])
\]
and 
\[
\|g^* (A, \Phi) - (A_0, 0) \|_{L^2_{k, \delta} (W[-\infty, 0] )} \leq C'\|(g')^* (A, \Phi) - (A_0, 0) \|_{L^2_{k, \delta} (W[-\infty, 0] )}\leq C'D(C). 
\]
This gives the desired result. 
\end{proof}

The topological energy $\mathcal{E}^{\text{top}}$ and the analytic energy $\mathcal{E}^{\text{top}}$ for configurations on $\hat{W}[-\infty, 0]$ are defined along the book by Kronheimer--Mrowka~\cite[Definition 4.5.4]{KM07}.
Note that, for a configuration $(A, \Phi)$ converging to $(A_0,0)$ on the periodic end, the boundary terms in the topological energy corresponding to the end vanishes, while the boundary terms corresponding to the cylindrical end may survive.
If such a configuration $(A, \Phi)$ is a ${W}[-\infty, 0]$-trajectory and is asymptotic to $\mathfrak{c}$ on the cylindrical end, we have that
\begin{align}
\label{eq: top en and csd}
\mathcal{E}^{\text{top}}(A, \Phi) = C_{X} - CSD(\mathfrak{c}),
\end{align}
where $C_{X}$ depends only on $X$ and the fixed metric and spin structure on $X$.
Moreover, we have that $\mathcal{E}^{\text{top}}(A, \Phi) = \mathcal{E}^{\text{an}}(A, \Phi)$ as well as for a configuration over a compact $4$-manifold.


\begin{proof}[Proof of \cref{bounded}]
Let $\delta_3$ be the constant given in \cref{exp}.
Suppose that  
\[
(x, y ) \in  \mathcal{U}_{k, \delta} \oplus ( \operatorname{Map}  ([0,\infty) ,  L^2_{k-\frac{1}{2} } (i\Lambda^1_Y )  \oplus L^2_{k-\frac{1}{2} } (S^+_{Y}  ) )
\]
satisfies the assumption of \cref{bounded}.
First, we state a pasting lemma: 

\begin{lem}
\label{lem: Tira gl}
The pair $(x, y )$ gives rise to a finite type $W[-\infty, 0]$-trajectory $(A , \Phi)$. 
\end{lem}
\begin{proof}This is essentially the same as the proof of \cite[Corollary~4.3]{Khan15}.
\end{proof}

It follows from \cref{lem: Tira gl} that we have a solution $(A, \Phi)$  to the Seiberg--Witten equations on $\hat{W}[-\infty, 0]$ whose topological energy is finite.

Recall that the set of critical points of $CSD$ modulo gauge is compact.
Since we consider a spin structure now, $CSD$ is gauge invariant.
Therefore the set of critical values of $CSD$ is compact.

Since we have assumed that $ | \lim_{t\to \infty }  CSD ( y(t) ) | < \infty$, we have that
\[
|CSD(y(t)) - CSD(y(t+1))| \to 0
\]
as $t \to \infty$, and therefore there exists a critical point of $CSD$ to which $(A, \Phi)$ is $L^2_{k-\frac{1}{2}}$-asymptotic as $t \to \infty$.
This combined with \eqref{eq: top en and csd} implies that $\mathcal{E}^{\text{top}} (A, \Phi)$ is uniformly bounded, and hence so is $\mathcal{E}^{\text{an}}(A, \Phi)$.

We claim that the analytic energy of $(A, \Phi)$ restricted to $W[-\infty,-1]$ is also uniformly bounded.
To see this, let us decompose $\hat{W}[-\infty,-1]$ into three parts:
the periodic part $W[-\infty,-1]$, the cylindrical part $\R^{\geq 0} \times Y$, and the `joint' between the periodic part and the cylindrical part.
We have seen that the analytic energy of $(A, \Phi)$ on $\hat{W}[-\infty,-1]$ is uniformly bounded, and this energy is the sum of the energies on these three parts.
Therefore, to prove that the analytic energy of $(A, \Phi)$ restricted to $W[-\infty,-1]$ is also uniformly bounded,
 it suffices to show that all of the energies on these three parts are bounded from below.
 But this is obvious to recalling the definition of the analytic energy. (See the proof of \cite[Lemma~4.8]{Lin19}.)

This uniform boundedness enables us to apply \cref{exp}, and thus we obtain the boundedness of $  \| x\|_{L^2_{k, \delta} } < R $:
\[
\|(A, \Phi) - (A_0, 0) \|_{L^2_{k,\delta} ( W[-\infty, 0] ) } \leq R
\]
for any $\delta \in [0,\delta_3)$.
The remaining boundedness result $\|y(t)\|_{L^2_{k-\frac{1}{2}}} < R$ follows from the same argument for $X$-trajectories, where $X$ is a compact $4$-manifold bounded by $Y$.
See \cite[Corollary~4.3]{Khan15} for example.
\end{proof}

\section{Relative Bauer--Furuta type invariant} \label{section: Relative Bauer--Furuta type invariant}
\label{section: rel BF inv}
In this section, we construct a relative Bauer--Furuta type invariant for 4-manifolds with periodic end and boundary $W[-\infty, 0]$. We mainly follow the methods given by Manolescu \cite{Ma03} and Khandhawit \cite{Kha15}. 

We consider a finite-dimensional approximation of the map 
\[
\mathcal{F}_{W[-\infty, 0]}  : \mathcal{U}_{k,  \delta}
 \to  \mathcal{V}_{k-1,  \delta} \oplus V (Y). 
 \]
We fix a weight $ \delta \in (0, \infty)$ satisfying 
\[
\delta \leq \min \{\delta_0, \delta_1, \delta_2, \delta_3\} 
\]
in the rest of this paper, where $ \delta_i$ are the constants appeared in \cref{rem deltazero}, \cref{cohomology}, \cref{Slice}, and \cref{bounded}. 
  Take sequences of subspaces 
 \[
  \mathcal{V}_1 \subset  \mathcal{V}_2 \subset \cdots \subset  \mathcal{V}_{k-1,  \delta} \text{ and } V^{\lambda_1}_{-\lambda_1} \subset V^{\lambda_2}_{-\lambda_2} \subset \cdots   \subset V (Y)  
  \]
  such that 
  \begin{itemize}
  \item[(i)] $(\im L_{W[-\infty, 0]} + p^{\lambda_n}_{-\lambda_n}\circ r)^{\perp_{\mathcal{V}_{k-1,  \delta} \oplus V (Y)}  } \subset \mathcal{V}_n \oplus  V^{\lambda_n}_{-\lambda_n} (Y)  $ 
 \item[(ii)] the $L^2$-projection $P_n :  \mathcal{V}_{k-1, \al}  \oplus V (Y) \to \mathcal{V}_n \oplus  V^{\lambda_n}_{-\lambda_n} (Y)$ satisfies 
 \[
 \lim_{n \to \infty} P_n (v) =v 
 \]
 for any $ v \in \mathcal{V}_{k-1, \delta}  \oplus V (Y)$.
  \end{itemize}
  Then we define a sequence of subspaces 
  \[
  \mathcal{U}_n :=  (L_{W[-\infty, 0]}+ p^{\lambda_n}_{-\lambda_n}\circ r)^{-1} ( \mathcal{V}_n \oplus  V^{\lambda_n}_{-\lambda_n} ). 
  \]
 This gives a family of the approximated Seiberg--Witten maps
  \[
  \{  \mathcal{F}_n : =  P_n (L_{W[-\infty, 0]}+C_{W[-\infty, 0]}, p^{\lambda_n}_{-\lambda_n} \circ r ) \colon   \mathcal{U}_n \to \mathcal{V}_n \oplus  V^{\lambda_n}_{-\lambda_n} (Y) \} . 
  \]
The following proposition gives us a well-defined continuous map between spheres. 
 \begin{prop}\label{cohomotopy1}For a large $n$ and a large positive real number $ R$, 
 there exists an index pair $(N_n,  L_n)$ of $V^{\lambda_n}_{-\lambda_n} (Y) $ and a sequence $\{\varepsilon_n\}$ of positive numbers such that 
 \begin{align}\label{cohomotopy}
 B( \mathcal{U}_n ; R) / S( \mathcal{U}_n ; R)  \to ( \mathcal{V}_n/  B(\mathcal{V}_n, \varepsilon_n)^c) \wedge (N_n/ L_n) 
 \end{align}
 is well-defined, where $B( V ; R)$ is the closed ball in $V$ with radius $R$ and $S( V ; R)$ is the sphere in $V$ with radius $R$.
 \end{prop}

 For the proof of \cref{cohomotopy1}, we use the following proposition. 
  \begin{prop}\label{fin app conv}
Let $\{x_n\}$ be a bounded sequence in $\mathcal{U}_{k,  \delta}$ such that 
\[
(L_{W[-\infty, 0]}(x_n), p^{\lambda_n }_{-\infty} \circ r (x_n ) )\in \mathcal{V}_n \times V^{\lambda_n}_{-\lambda_n} 
\]
and 
\[
P_n (L_{W[-\infty, 0]}+C_{W[-\infty, 0]}) x_n \to 0 . 
\]
Let $y_n: [0, \infty) \to V^{\lambda_n}_{-\lambda_n}$ be a uniformly bounded sequence of trajectories such that 
\[
y_n(0) = p^{\lambda_n}_{-\infty} \circ r (x_n) . 
\]
Then, after taking a subsequence, $\{x_n\}$ converges to a solution $x \in \mathcal{U}_{k,  \delta}$ (in the topology of $\mathcal{U}_{k,  \delta}$) and $\{y_n(t)\}$ converges to $y(t) (\forall t\in [0,\infty))$ in $L^2_{k-\frac{1}{2}}$ which is a solution of the Seiberg--Witten equations on $\R^{\geq 0}\times Y$. 
\end{prop}
\begin{proof}
The proof is similar to the proof of \cite[Proposition 3]{Kha15}. By the same argument, one sees the following result: for any compact set $I \subset (0, \infty)$, after taking a subsequence, $y_n(t)$ uniformly converges to $y(t)$ in $L^2_{k-\frac{1}{2}}$, where $y(t)$ is the weak limit.

For the sequence $\{x_n\}$, we need to ensure: 
\begin{itemize}
\item after taking a subsequence, $p^0_{-\infty} y_n(0) \to p^0_{-\infty} r (x)$ in $L^2_{k-\frac{1}{2}}$, where $x$ is the weak limit and 
\item after taking a subsequence, the sequence $\{x_n\} $ converges to $x$ in $L^2_{k, \delta} (X)$. 
\end{itemize} 
 The proof of the second statement is the only difference between our construction and the usual Bauer--Furuta invariant. Here we again follow the method given in \cite{IT20}. 
 To obtain the convergence of $ \{x_n\}$, we will use the following inequality obtained by the Fredholm property of $L_{W[-\infty, 0]}$: 
 there exists a constant $C>0$ such that, for any $x \in \mathcal{U}_{k, \delta}$, 
 \[
 \| x\|_{L^2_{k, \delta} } \leq  C ( \|  L_{W[-\infty, 0]} (x) \| _{ L^2_{k-1 , \delta}  } + \| p^0_{-\infty} r (x) \|_{L^2_{k-\frac{1}{2}} }  + \| x\|_{L^2})  .
 \]
 Then, by the same discussion given in the proof of \cite[Lemma 3.18]{IT20}, we complete the proof. 
 \end{proof} 
 
\begin{proof}[Proof of \cref{cohomotopy1}]
We combine \cref{fin app conv}, \cref{bounded} and the proof of \cite[Proposition 4.5]{Kha15} and complete the proof. 
\end{proof}
 
By \cref{cohomotopy1}, we obtain a family of the continuous maps \eqref{cohomotopy}. By the definition of Fredholm index,  we have 
 \[
\ind_\R(L_{W[-\infty, 0]}\oplus p^{\lambda_n}_{-\infty} \circ r)= \dim_\R \mathcal{U}_n - \dim_\R \mathcal{V}_n - \dim_\R V^{\lambda_n}_{-\lambda_n} . 
 \]

We obtain a map stably written by 
\[
\Psi  :  (\wt{\R}^m \oplus \quat^n )^+  \to  (\wt{\R}^{m'} \oplus \quat^{n'})^+ \wedge \Sigma^{-V^0_{-\lambda_n}  }(N_n/ L_n), 
\]
here we fixed trivializations of vector spaces. 
\begin{rem}
Our construction gives an invariant of 4-manifolds with periodic end admitting periodic PSC metric on the end. 
This can be regarded as relative Bauer--Furuta invariant corresponding to \cite{Ve14}.  
\end{rem}

\section{The proof of \cref{real main}} \label{section: The proof of main}
\label{section: proof of main}

In this \lcnamecref{section: proof of main}, we prove \cref{real main}.
Recalling the definition of local equivalence \cite{Sto20}, what we have to do is to construct a certain type of map called {\it local map} from $\SWF(Y,\frakt)$ to $\left[ \left( S^{0},0,-\lambda_{SW} (X, \s)/2\right) \right]$, and also a local map from $\left[ \left( S^{0},0,-\lambda_{SW} (X, \s)/2\right) \right]$ to $\SWF(Y,\frakt)$.

We shall consider the relative Bauer--Furuta invariant on the `left-periodic' manifold $W[-\infty, 0]$ and that on `right-periodic' manifold $W[0,\infty]$.
These two relative Bauer--Furuta invariants give the desired two local maps.

\begin{proof}[Proof of \cref{real main}]
In \cref{section: rel BF inv},
under the assumption of the existence of PSC metric on $X$, we constructed a $\Pin(2)$-equivariant continuous map of the form
\begin{align}
\label{eq: rel BF basic}
f : (\tilR^{m_{0}} \oplus \quat^{n_{0}})^{+} \to (\tilR^{m_{1}} \oplus \quat^{n_{1}})^{+} \wedge I_{-\lambda}^{\lambda}
\end{align}
as the relative Bauer--Furuta invariant over $W[-\infty, 0]$.
One sees that $f^{S^{1}}$ induces a $\Pin(2)$-homotopy equivalence by \cref{cohomology}. The numbers $m_{0}-m_{1}, n_{0}-n_{1}$ are given by
\begin{align}
\begin{split}
\label{eq: m0m1n0n1}
m_{0} -m_{1} &= \dim V^{0}_{-\lambda}(\R),\\
2(n_{0} -n_{1}) &= \ind_{\C}D^+_{W[-\infty,0]} + \dim_{\C} V^{0}_{-\lambda}(\quat)\\
&= \lambda_{SW}(X, \fraks) + n(Y,\frakt,g) + \dim_{\C} V^{0}_{-\lambda}(\quat). 
\end{split}
\end{align}
For the notations $V^{0}_{-\lambda}(\R)$ and $ V^{0}_{-\lambda}(\quat)$, see \eqref{decom}.
Here we have used \cref{prop: excision} to get the second equality of \eqref{eq: m0m1n0n1} and \cref{cohomology} to get the first equality. 

Equations \eqref{eq: rel BF basic} and \eqref{eq: m0m1n0n1} mean that the map $f$ gives a local map from $\left[ \left( S^{0},0,-\lambda_{SW} (X, \s)/2\right) \right]$ to $\SWF(Y,\frakt)$.

Next, instead of the `left-periodic' manifold $W[-\infty, 0]$,
we consider the `right-periodic' manifold
\[
W[0, \infty] = W \cup_{Y} W\cup_{Y} W \cup_{Y} \cdots.
\]
Repeating analysis in \cref{section: rel BF inv} for $W[0, \infty]$ instead of $W[-\infty, 0]$,
we obtain a $\Pin(2)$-map of the form
\begin{align}
\label{eq: rel BF basic'}
f' : (\tilR^{m'_{0}} \oplus \quat^{n'_{0}})^{+} \to (\tilR^{m'_{1}} \oplus \quat^{n'_{1}})^{+} \wedge \bar{I}_{-\lambda}^{\lambda}
\end{align}
as the relative Bauer--Furuta invariant over $W[0,\infty]$.
Here $\bar{I}_{-\lambda}^{\lambda}$ denotes the Conley index for $-Y$.
As well as $f$ above, $(f')^{S^{1}}$ induces a $\Pin(2)$-homotopy equivalence by \cref{cohomology1}.
For $\mu \leq 0 \leq \lambda$,
as in \cite[Proof of Proposition~3.8]{Ma16},
let us denote by $\bar{V}^{\lambda}_{\mu}$ the vector space $V^{\lambda}_{\mu}$ defined for $-Y$.
Note that, for $\mu < 0 < \lambda$, we have an identification $\bar{V}^{\lambda}_{\mu} \cong V^{-\mu}_{-\lambda}$, and in particular
$\bar{V}^{\lambda}_{-\lambda} \cong V^{\lambda}_{-\lambda}$.
Under this notation, $m'_{0}-m'_{1}, n'_{0}-n'_{1}$ are given by
\begin{align}
\begin{split}
\label{eq: m0m1'eq: n0n1'}
m'_{0} -m'_{1} &= \dim \bar{V}^{0}_{-\lambda}(\R),\\
2(n'_{0} -n'_{1}) &= \ind_{\C}D^+_{W[0,\infty]} + \dim_{\C} \bar{V}^{0}_{-\lambda}(\quat).
\end{split}
\end{align}
By an argument using a duality map as in \cite[page 168]{Ma16},
we obtain a $\Pin(2)$-map
\begin{align*}
f'' : (\tilR^{m'_{0}} \oplus \quat^{n'_{0}})^{+} \wedge I_{-\lambda}^{\lambda} \to (\tilR^{m'_{1}} \oplus \quat^{n'_{1}})^{+} \wedge (V^{\lambda}_{-\lambda})^{+}
\end{align*}
from \eqref{eq: rel BF basic'}.  
The vector space $V^{\lambda}_{-\lambda}$ can be decomposed so that $V^{\lambda}_{-\lambda}(\R) \oplus V^{\lambda}_{-\lambda}(\quat)$.
Set 
\begin{align}
\begin{split}
\label{eq: prf main double pr}
&m_{1}'' = m_{1}' + \dim V^{\lambda}_{-\lambda}(\R),\\
&n_{1}'' = n_{1}' + \dim_{\quat} V^{\lambda}_{-\lambda}(\quat).
\end{split}
\end{align}
Then the domain and codomain of $f''$ are given by
\begin{align}
\label{eq: rel BF basic''}
f'' : (\tilR^{m'_{0}} \oplus \quat^{n'_{0}})^{+} \wedge I_{-\lambda}^{\lambda} \to (\tilR^{m''_{1}} \oplus \quat^{n''_{1}})^{+}.
\end{align}
We shall show that $f''$ gives a local map from $\SWF(Y,\frakt)$ to $\left[ \left( S^{0},0,-\lambda_{SW} (X, \s)/2\right) \right]$.
The restriction $(f'')^{S^1}$ is a $Pin(2)$-homotopy equivalence since $f'$ is so. 
One may assume $\lambda$ was taken to avoid the eigenvalues of the linearization $l$ of the flow equations.
Then we have 
\begin{align}
\begin{split}
\label{eq: zero eig proof of main}
&\bar{V}_{-\lambda}^{0}(\R) = V_{0}^{\lambda}(\R),\\
& \bar{V}_{-\lambda}^{0}(\quat) = V_{0}^{\lambda}(\quat) \oplus \ker \D.
\end{split}
\end{align}
Here, to obtain the first equality, we have used $\Ker (\ast d : \ker d^{\ast} \to \Omega^{1}(Y))=0$ deduced from the assumption that $b_{1}(Y)=0$.
Using \eqref{eq: zero eig proof of main}, we have 
\begin{align}
\begin{split}
\label{eq: Vr proof of main}
V^{\lambda}_{-\lambda}(\R) 
&\cong  V^{0}_{-\lambda}(\R) \oplus V^{0}_{-\lambda}(\R)\\
&\cong  V^{0}_{-\lambda}(\R) \oplus \bar{V}_{-\lambda}^{0}(\R)
\end{split}
\end{align}
and 
\begin{align}
\begin{split}
\label{eq: Vh proof of main}
V^{\lambda}_{-\lambda}(\quat) \oplus \ker \D
&\cong  V^{0}_{-\lambda}(\quat) \oplus V^{0}_{-\lambda}(\quat) \oplus \ker \D\\
&\cong  V^{0}_{-\lambda}(\quat) \oplus \bar{V}_{-\lambda}^{0}(\quat).
\end{split}
\end{align}
Combining \eqref{eq: m0m1'eq: n0n1'} with \eqref{eq: prf main double pr}, \eqref{eq: Vr proof of main} and \eqref{eq: Vh proof of main}, we obtain
\begin{align}
\begin{split}
\label{eq: m0'-m1''}
m_{0}'-m_{1}''
&= m_{0}'-m_{1}' - \dim \bar{V}^{\lambda}_{-\lambda}(\R) \\
&= \dim \bar{V}^{0}_{-\lambda}(\R) - \dim V^{\lambda}_{-\lambda}(\R) = -\dim V^{0}_{-\lambda}(\R),
\end{split}
\end{align}
\begin{align}
\begin{split}
\label{eq: n0'-n1''}
n_{0}'-n_{1}''
&= n_{0}'-n_{1}' - \dim \bar{V}^{\lambda}_{-\lambda}(\quat) \\
&= \dim \bar{V}^{0}_{-\lambda}(\quat) - \dim V^{\lambda}_{-\lambda}(\quat)\\
&= \ind_{\quat}D^+_{W[0,\infty]} 
+ \dim \ker \D - \dim_{\quat} V^{0}_{-\lambda}(\quat).
\end{split}
\end{align}

Let us calculate $\ind_{\quat}D^+_{W[0,\infty]}$ in the last equality.
Let $M'$ be an oriented compact smooth $4$-manifold with boundary $\del M' = Y$.
Set
\[
M'_{\infty} =  M' \cup_{Y} W\cup_{Y} W \cup_{Y} \cdots.
\]
Then, as well as \cref{prop: excision2}, we obtain
\begin{align}
\label{eq: excision2}
\ind_{\C}D^+_{W[0,\infty]}  + \ind_{\C}D^+_ {M'} + \dim_{\C} \Ker\D = \ind_{\C}D^+_ {M'_{\infty}},
\end{align}
On the other hand, for a PSC metric $g$ on $X$, we have
\begin{align}
\label{Jimila3'}
  \lambda_{SW}(X,\fraks) 
 = -w(X,g,0) 
 = -\ind_{\C} D^+_{M'_{\infty}} - \frac{\sigma(M')}{8}.
\end{align}
Recalling the definition of $n(Y,\frakt,g)$, we have
\begin{align}
\label{def n(Y)}
n(Y,\frakt,g)
 = \ind_{\C} D^+_{M'} + \frac{\sigma(M')}{8}.
\end{align}
Combining \eqref{eq: excision2} with \eqref{Jimila3'} and \eqref{def n(Y)}, we have
\begin{align}
\label{eq: ind W0inf}
\ind_{\quat}D^+_{W[0,\infty]}
= -\frac{1}{2}(\lambda_{SW}(X,\fraks) + n(Y,\frakt,g)) -\dim_{\quat}\ker\D.
\end{align}

It follows from \eqref{eq: n0'-n1''} and \eqref{eq: ind W0inf} that
\begin{align}
\begin{split}
\label{eq: n0'-n1''2}
n_{0}'-n_{1}''
= -\frac{1}{2}(\lambda_{SW}(X,\fraks) + n(Y,\frakt,g)) - \dim_{\quat} V^{0}_{-\lambda}(\quat).
\end{split}
\end{align}
Now we deduce from \eqref{eq: rel BF basic''}, \eqref{eq: m0'-m1''}, and \eqref{eq: n0'-n1''2} that
$f''$ gives a local map from $\SWF(Y,\frakt)$ to $\left[ \left( S^{0},0,-\lambda_{SW} (X, \s)/2\right) \right]$.
\end{proof}

\section{Obstruction to embeddings of 3-manifolds into 4-manifolds with PSC metric}
\label{section: Obstruction to embedding of 3-manifolds into 4-manifolds with PSC metric}

\cref{real main} gives an obstruction to embedding of 3-manifolds into 4-manifolds with PSC metric under a homological assumption. By a standard surgery argument enables us to prove the following generalization of \cref{real main}. 
\begin{theo}\label{emb}
Let $(X, \s)$ be an oriented spin closed connected 4-manifold with $b_2(X)=0$ and $Y$ a smooth oriented closed codimension-1 submanifold of $X$. Suppose $b_1(Y)=0$ and $X$ admits a PSC metric. Then the local equivalence class of $\SWF(Y,\frakt)$ is given by
\begin{align}
[\SWF(Y,\frakt)] 
= \left[ \left( S^{0},0,-\frac{\delta (Y , \mathfrak{t}) }{2}\right) \right], 
\end{align}
where $\mathfrak{t}:= \mathfrak{s}|_Y$. 
\end{theo}
This theorem can be seen as a Seiberg--Witten analogue of the result proven by Yang--Mills gauge theory \cite[Theorem 1.9]{T19}.  Using the Heegaard Floer correction term, Levine--Ruberman \cite{LR19} gave an obstruction of codimension-1 smooth embeddings into homology $S^1\times S^3$'s. For the obstructions to codimension-1 smooth embeddings into indefinite spin 4-maniolds, see \cite{PMK17}.

\begin{proof}
We argue the case that $[Y]\neq 0$ and that $[Y] = 0$ individually. First, let us assume $[Y]\neq 0$. In this case, the cobordism $W_0 := \overline{X \setminus Y}$ from $Y$ to itself is connected. When $b_2(W_0)=0$, one can see $X$ is a rational homology $S^1\times S^3$ and $ [Y]$ generates $H_3(X)$. Thus, by \cref{real main}, one has 
\[
[\SWF(Y,\frakt)] = \left[ \left( S^{0},0,-\frac{\lambda_{SW} (X, \s)}{2}\right) \right]=\left[ \left( S^{0},0,-\frac{\delta (Y , \mathfrak{t}) }{2}\right) \right] .
\]
When $b_2(W_0)>1$, we take disjoint simple closed curves $l_1, \cdots , l_{b_2(W_0)}$ in $X$ which generate $H_2(W_0; \Z)$. 
We extend $l_1, \cdots , l_{b_2(W_0)}$ to disjoint smooth embeddings from $S^1\times D^3$'s into $W_0$ and denote them by the same notations. 
We consider the manifold 
\[
W_0({l_1, \cdots , l_{b_2(W_0)}} )
\]
obtained by the surgery of $W_0$ along $l_1\cup  \cdots \cup  l_{b_2(W_0)}$. One can see $W_0({l_1, \cdots , l_{b_2(W_0)}} ) $ also admits a spin structure. We write the glued manifold along the boundary of $W_0({l_1, \cdots , l_{b_2(W_0)}} ) $ by $X({l_1, \cdots , l_{b_2(W_0)}} )$.

Since we are considering codimension-3 surgeries, \cite[Theorem A]{GL80} implies that $X({l_1, \cdots , l_{b_2(W_0)}} )$ also admits a PSC metric. 
 The manifold $X({l_1, \cdots , l_{b_2(W_0)}} )$ is a spin rational homology $S^1\times S^3$. By construction, $Y$ is smoothly embedded into $X({l_1, \cdots , l_{b_2(W_0)}} )$ such that 
 \[
0  \neq [Y] \in H_3(X({l_1, \cdots , l_{b_2(W_0)}} ); \Z)\cong \Z .
 \]
 An easy observation shows that $[Y]$ generates $H_3(X({l_1, \cdots , l_{b_2(W_0)}} ); \Z)$. Thus one can use \cref{real main} and see 
 \[
[\SWF(Y,\frakt)] = \left[ \left( S^{0},0,-\frac{\lambda_{SW} (X({l_1, \cdots , l_{b_2(W_0)}} ) )}{2}\right) \right]=\left[ \left( S^{0},0,-\frac{\delta (Y , \mathfrak{t}) }{2}\right) \right] .
\]

Next, we consider the case $[Y]=0$. 
In this case, our cobordism $W_0$ should have two connected components: $W_{0}^+ \cup W_0^-$. Suppose $\partial W_{0}^+ =Y$ and $\partial W_0^- = -Y$. By 1-handle surgery, one can assume that $W_{0}^+$ and  $W_0^-$ are spin rational homology $D^4$'s.
Thus the relative Bauer--Furuta invariants $BF_{W_{0}^+}$ and $BF_{W_{0}^-}$ gives rise to the local equivalence between $SWF(Y, \mathfrak{t})$ and $\left[ \left( S^{0},0, 0 \right) \right]$. 
This completes the proof. 
\end{proof}

\begin{cor}\label{emb1}
Let $Y$ be an integral homology $3$-sphere. Suppose that at least two of $\alpha(Y), \beta(Y), \gamma(Y), \delta(Y),\overline{\delta}(Y), \underline{\delta}(Y), \kappa(Y)$ do not coincide with each other. 
Then $Y$ does not admit any smooth embedding into a spin closed 4-manifold with a PSC metric satisfying $b_2(X)=0$. 
 \end{cor}
 
 \begin{rem}
 Freedman's result (\cite{F82}) implies that all homology 3-spheres have a locally flat embedding into $S^4$, and \cref{emb1} is false for locally flat topological embeddings. 
 \end{rem}

\section{Examples}
\label{section: Examples}

In this section we use \cref{cor: glued} to obtain a concrete family of 4-manifolds which does not admit PSC metrics. 
In order to use \cref{cor: glued}, we need to calculate the homology cobordism invariants $\alpha$, $\beta$, $\gamma$, $\delta$. 
The following remark gives a method to calculate $\delta$ for a large class of 3-manifolds: 
\begin{rem}In \cite[Remark 1.1]{LRS18}, it is mentioned that Heegaad Floer correction term $d(Y, \s)$ and the monopole Fr\o yshov invariant $h(Y,\s) $ satisfy
\[
d(Y, \s) = -2 h(Y,\s) , 
\]
for any spin$^c$ rational homology 3-sphere $(Y,\s)$. 
Moreover, it is proved in \cite{LM18} that
\begin{align*}
-h ( Y, \mathfrak{s}  ) =  \delta (Y, \mathfrak{s}).
\end{align*}
Therefore one can use calculations of correction terms in Heegaard Floer theory (\cite{OS03, BN13, Tw13, KS19}) in order to calculate $\delta (Y, \mathfrak{s})$. 
\end{rem}

For the invariants $\alpha$, $\beta$ and $\gamma$,
we mainly use Stoffregen's computation results \cite{Sto20} for Seifert homology 3-spheres and connected sums of them. 

Before considering to the connected sum, we start with a single Seifert homology 3-sphere.  The following result is proved by Stoffregen \cite{Sto20}.
Recall that a Seifert rational homology 3-sphere $Y$ is called {\it negative} if the underlying orbifold line bundle of $Y$ is of negative degree (see \cite[Section 5]{Sto20}).

\begin{theo}[\cite{Sto20}]\label{single} The following results hold.
\begin{itemize}
\item[(i)] 

Let $Y$ be a Seifert homology 3-sphere with negative fibration.  Then 
\[ \beta (Y) = \gamma(Y) = -\overline{\mu} (Y), \ \text{ and } 
\]
\[
\alpha (Y) = \begin{cases} d(Y)/2 = \delta (Y) \text{ if } d(Y)/2  \equiv  -\overline{\mu} (Y) \operatorname{mod} 2  \\ 
d(Y)/2+1 = \delta (Y)+ 1 \text{ otherwise }
\end{cases} 
\] 
hold. 
\item[(ii)]
Let $Y$ be a Seifert homology 3-sphere with positive fibration.  Then 
\[ \alpha (Y) = \beta(Y) = -\overline{\mu} (Y), \ \text{ and } 
\]
\[
\gamma (Y) = \begin{cases} d(Y)/2 = \delta (Y) \text{ if } d(Y)/2  \equiv  -\overline{\mu} (Y) \operatorname{mod} 2  \\ 
d(Y)/2-1 = \delta (Y)-  1 \text{ otherwise }
\end{cases} 
\] 
hold. 
\end{itemize} 
\end{theo}

Combining \cref{main cor} with \cref{single}, we obtain:
\begin{theo}
Let $Y'$ be a Seifert homology 3-sphere such that
\[
 -\overline{\mu} (Y') \neq \delta(Y' ) , 
 \]
 where $\overline{\mu}$ is the Neumann--Siebenmann invariant for graph homology 3-spheres introduced in \cite{N80,Si80}. 
 Let $Y$ be an oriented homology 3-sphere which is homology cobordant to $Y'$.
 Then, for any homology cobordism $W$ from $Y$ to itself, the 4-manifold obtained from $W$ by gluing the boundary components does not admit a PSC metric.   
\end{theo}

The invariant $\overline{\mu}$ has a concrete recursion formula for $\Sigma(a_1, \cdots, a_n)$. See \cite[(2.8), (2.9) in Subsection 2.4.2]{Sa02}. Although (2.8) and (2.9) in \cite[Subsection 2.4.2]{Sa02} are formulae for the Rochlin invariant, it is pointed out in \cite[page 197]{Sa02} that the same formula holds also for the invariant $\overline{\mu}$.
We also note another way to compute $\overline{\mu}$ based on the $w$-invariant.
For the definition of $w$-invariant, see \cite[Definition 2.2]{Fuk09}.
In \cite{Fuk09}, the $w$-invariants of several types of Seifert homology 3-spheres are computed, and
 the following relation is given in \cite{Sa07, FFU01, Fuk00}: for any Seifert homology 3-sphere of type $\Sigma(2, q,r)$,
\[
w ( \Sigma (2, q,r), X(2,q,r ), \s) = - \overline{\mu} (\Sigma(2, q,r)) . 
\]
Here $(X(2,q,r ), \s)$ is a certain spin 4-orbifold.
For the unique way to construct $X(2,q,r)$, see the sentences after \cite[Theorem 3.1]{Fuk09}. 

Also, in \cite{Sto20}, there are direct computations of $\alpha$, $\beta$ and $\gamma$. Using them, we can prove: 
\begin{cor}Suppose a homology 3-sphere $Y$ is homology cobordant to one of Seifert homology 3-spheres with types: 
\begin{align*}
 (2,3,12k-1), (2,3,12n+7) , (2,5,20k+11), (2,5,20k-1), (2,5,20k-3),\\
 (2,5,20k+13), (2,7,28k-1), (2,7,28k+15), (2,7, 28k-3) , (2,7, 14k+3),\\
 \text{ and }(2,7, 14k-5) . 
\end{align*}
 Then, for any homology cobordism $W$ from $Y$ to itself, the 4-manifold obtained from $W$ by gluing the boundary components does not admit a PSC metric.   
\end{cor} 
\begin{proof}
We just combine computation results \cite{Fuk09, Ma16, Sto20} of $\alpha, \beta, \gamma$ and $\overline{\mu}$ and \cref{cor: glued}. 
\end{proof} 
\begin{rem}
We remark that for homology $S^1\times S^3$'s obtained as mapping tori, enlargeable obstruction \cite{GL80} can be used to obstruct PSC metrics. 
A large class of homology $S^1\times S^3$ which are not obtained as mapping tori are introduced in \cite[Subsection 4.4.1]{KT20}. 
Also, a review of several known obstructions for homology $S^1\times S^3$'s is given in \cite[Subsection 4.4]{KT20}. 
\end{rem}

Next, we consider the connected sums of Seifert homology 3-spheres. 
In order to obtain a certain connected sum formula of invariants $\alpha$, $\beta$ and $\gamma$ for Seifert homology 3-spheres, Stoffregen considered a class of Seifert homology 3-spheres, called projective type.
We call a negative Seifert rational homology 3-sphere $Y$ with a spin structure $\s$ {\it projective} if its Heegaard Floer homology is of the form
\[
HF^+ ( Y, \s) \cong \mathcal{T}^+_d \oplus \mathcal{T}_{-2n+1}^+ (m) \oplus \bigoplus_{i\in I} \mathcal{T}_{a_i}^+ (m_i)^{\oplus 2} 
\]
for some $n$, $m$, $d $, $a_i$, $m_i $ and some index set $ I$, where 
\begin{itemize}
\item $\mathcal{T}^+ := \mathbb{F} [U, U^{-1}] / \mathbb{F} [U] $, where $\mathbb{F}$ is the field of two elements, 
\item $\mathcal{T}^+(i) :=  \mathbb{F} [U^{-i+1}, U^{-i+2}] / \mathbb{F} [U] $, and 
\item $\mathcal{T}^+_d (n) := \mathcal{T}^+(n)$ whose grading is shifted by $-d$.
\end{itemize} 

 There are many examples of projective Seifert homology 3-spheres \cite{Ne07, BN13, Tw13}. 
It is confirmed in \cite{Ne07, BN13, Tw13} that $ \Sigma(p,q,pqk\pm1)$ is projective for a relatively prime pair $(p,q)$ and positive integer $k$.

\begin{theo}[\cite{Sto20}]  \label{St}
Let $Y_1, \cdots, Y_n$ be negative Seifert homology 3-spheres of projective type. Suppose $\delta(Y_1) \leq \cdots \leq \delta(Y_n)$. Set $\wt{\delta}_i := \delta(Y_i) + \overline{\mu} (Y_i)$. Then 
 \begin{itemize}
 \item  $\alpha(Y_1 \# \cdots \# Y_n) =   2\lfloor \frac{\sum_{i=1}^n \wt{\delta}_i +1}{2} \rfloor  -\sum_{i=1}^n \overline{\mu} ( Y_i)  
$
 \item $\beta(Y_1 \# \cdots \# Y_n)  = 2\lfloor \frac{\sum_{i=1}^{n-1} \wt{\delta}_i +1}{2} \rfloor  -\sum_{i=1}^n \overline{\mu} ( Y_i)  
 $
 \item $\gamma(Y_1 \# \cdots \# Y_n)  = 2\lfloor \frac{\sum_{i=1}^{n-2} \wt{\delta}_i +1}{2} \rfloor -\sum_{i=1}^n \overline{\mu} ( Y_i)  
$
\end{itemize} 
\end{theo} 
Combining \cref{main cor} with \cref{St}, we can prove: 
\begin{theo} \label{app2}
 Let $Y_1, \cdots, Y_n$ be negative Seifert homology 3-spheres of projective type. Suppose $\delta(Y_1) \leq \cdots \leq \delta(Y_n)$. Set $\wt{\delta}_i := \delta(Y_i) + \overline{\mu} (Y_i)$. Suppose that at least two of the following four integers are distinct:
 \begin{align*}
 \sum_{i=1}^n {\delta} ( Y_i), \quad
2 \lfloor \frac{\sum_{i=1}^n \wt{\delta}_i +1}{2}\rfloor -\sum_{i=1}^n \overline{\mu} ( Y_i),\\
 2\lfloor \frac{\sum_{i=1}^{n-1} \wt{\delta}_i +1}{2} \rfloor  -\sum_{i=1}^n \overline{\mu} ( Y_i), \quad2\lfloor \frac{\sum_{i=1}^{n-2} \wt{\delta}_i +1}{2} \rfloor    -\sum_{i=1}^n \overline{\mu} ( Y_i). 
  \end{align*}
Let $Y$ be an oriented homology 3-sphere which is homology cobordant to $Y_1 \# \cdots \# Y_n$.
Then, for any homology cobordism $W$ from $Y$ to itself, the 4-manifold obtained from $W$ by gluing the boundary components does not admit a PSC metric. 
\end{theo} 

For a concrete family,  one can see the following non-existence of PSC metrics for connected sums: 
\begin{cor}
Suppose a homology 3-sphere $Y$ is homology cobordant to one of homology 3-spheres: 
\begin{itemize}
\item $ \#_j  \Sigma (2,3,12n-1)$, 
\item $ \#_j  \Sigma (2,5,20n-1)$, and 
\item $ \#_j  \Sigma (2,7,28n-1)$
\end{itemize}
for some $j \in \Z_{>0}$, where  $\#_j  Y$ means the connected sum of $j$-copies of $ Y$. 
Then, for any homology cobordism $W$ from $Y$ to itself, the 4-manifold obtained from $W$ by gluing the boundary components does not admit a PSC metric. 
\end{cor}

\begin{proof}
As it is calculated in \cite{Sto20}, one has 
\begin{itemize}
\item $\alpha (\Sigma (2,3,12n-1)) =2, \beta (\Sigma (2,3,12n-1)) =0, \gamma (\Sigma (2,3,12n-1)) =0,$
$ \overline{\mu}  (\Sigma (2,3,12n-1)) =1, \delta  (\Sigma (2,3,12n-1)) =1$, 
\item $\alpha (\Sigma (2,5,20n-1)) =2, \beta (\Sigma (2,5,20n-1)) =0, \gamma (\Sigma(2,5,20n-1)) =0,$
$ \overline{\mu}  (\Sigma (2,5,20n-1)) =1, \delta  (\Sigma (2,5,20n-1)) =1$, and
\item $\alpha (\Sigma (2,7,28n-1)) =2, \beta (\Sigma (2,7,28n-1)) =0, \gamma (\Sigma(2,7,28n-1)) =0,$
$ \overline{\mu}  (\Sigma (2,7,28n-1)) =1, \delta  (\Sigma (2,7,28n-1)) =2$. 
\end{itemize} 
Since $\Sigma(p,q,pq\pm1)$ are projective, it follows from from \cref{St} that
\begin{itemize}
\item $\alpha (\Sigma (2,3,12n-1)) =2\lfloor \frac{j+1}{2} \rfloor , \beta (\Sigma (2,3,12n-1)) =2\lfloor \frac{j}{2} \rfloor,$

$\gamma (\Sigma (2,3,12n-1)) =2\lfloor \frac{j-1}{2} \rfloor, \delta  (\#_j  \Sigma (2,3,12n-1)) =j$, 

\item $\alpha (\Sigma (2,5,20n-1)) =2\lfloor \frac{j+1}{2} \rfloor , \beta (\Sigma (2,5,20n-1)) =2\lfloor \frac{j}{2} \rfloor,$

$ \gamma (\Sigma(2,5,20n-1)) =2\lfloor \frac{j-1}{2} \rfloor,\delta  (\#_j  \Sigma (2,5,20n-1)) =j$, and
\item $\alpha (\Sigma (2,7,28n-1)) =2\lfloor \frac{2j+1}{2} \rfloor, \beta (\Sigma (2,7,28n-1)) =2j $,
$\gamma (\Sigma(2,7,28n-1)) =2\lfloor \frac{2j-1}{2} \rfloor,$
$  \delta  (\#_j  \Sigma (2,7,28n-1)) =j$. 
\end{itemize} 
Therefore, in these cases, the assumptions of \cref{app2} are satisfied, and \cref{app2} implies the desired conclusion. 
\end{proof}

\begin{rem}  We expect that the connected Seiberg--Witten Floer homology $SWFH_{\rm conn}(Y, \s)$ introduced in \cite{Sto20} can be used to obstruct PSC metrics. Also, the equivariant KO-theoretic homology cobordism invariants introduced in \cite{Lin15} should give another obstruction.  
\end{rem}
  
\bibliographystyle{plain}
\bibliography{tex}

\end{document}